\documentclass[a4paper,12pt]{article}

\usepackage{amsmath,amssymb,amsthm,amscd,enumerate}

\allowdisplaybreaks

\newcommand{\Fg}{\mathfrak{g}}
\newcommand{\Fh}{\mathfrak{h}}
\newcommand{\FH}{\mathfrak{H}}
\newcommand{\Fs}{\mathfrak{s}}
\newcommand{\Ft}{\mathfrak{t}}
\newcommand{\Fz}{\mathfrak{z}}
\newcommand{\FS}{\mathfrak{S}}
\newcommand{\FA}{\mathfrak{A}}

\newcommand{\BC}{\mathbb{C}}
\newcommand{\BR}{\mathbb{R}}

\newcommand{\BZ}{\mathbb{Z}}
\newcommand{\BF}{\mathbb{F}}

\newcommand{\CC}{\mathcal{C}}
\newcommand{\MCD}{\mathcal{D}}
\newcommand{\CL}{\mathcal{L}}
\newcommand{\CI}{\mathcal{I}}

\newcommand{\SC}{\mathsf{C}}

\newcommand{\ad}{\mathop{\rm ad}\nolimits}
\newcommand{\id}{\mathop{\rm id}\nolimits}

\newcommand{\Hom}{\mathop{\rm Hom}\nolimits}
\newcommand{\End}{\mathop{\rm End}\nolimits}
\newcommand{\Ind}{\mathop{\rm Ind}\nolimits}
\newcommand{\Aut}{\mathop{\rm Aut}\nolimits}

\newcommand{\rank}{\mathop{\rm rank}\nolimits}
\newcommand{\Sym}{\mathop{\rm Sym}\nolimits}
\newcommand{\Img}{\mathop{\rm Im}\nolimits}
\newcommand{\Ker}{\mathop{\rm Ker}\nolimits}
\newcommand{\Ni}{\mathop{\rm Ni}\nolimits}
\newcommand{\Co}{\mathop{\rm Co}\nolimits}

\newcommand{\ten}{\mathrm{X}}
\newcommand{\Con}[2]{\mathrm{Con}(#1;#2)}

\newcommand{\pair}[2]{\langle #1,\,#2 \rangle}

\newcommand{\Gone}{\bigl\langle \omega^{(6)} \bigr\rangle}

\newcommand{\sbg}{\preceq}

\newcommand{\ve}{\varepsilon}
\newcommand{\vp}{\varphi}

\newcommand{\ha}[1]{\widehat{#1}}
\newcommand{\ti}[1]{\widetilde{#1}}
\newcommand{\ol}[1]{\overline{#1}}

\newcommand{\kc}{\langle \kappa \rangle}

\newcommand{\bk}{{\bf k}}

\newcommand{\bmu}{{\boldsymbol \mu}}

\newcommand{\bqed}{\quad \hbox{\rule[-0.5pt]{3pt}{8pt}}}

\makeatletter
\@addtoreset{equation}{subsection}

\renewcommand\section{\@startsection{section}{1}{0pt}
{-3.5ex plus -1ex minus -.2ex}{1.0ex plus .2ex}{\large\bf}}
\renewcommand\subsection{\@startsection{subsection}{1}{0pt}
{2.5ex plus 1ex minus .2ex}{-1em}{\bf}}
\makeatother

\newcommand{\svsp}{\vspace{1.5mm}}
\newcommand{\vsp}{\vspace{3mm}}

\theoremstyle{plain}
\newtheorem{thm}{Theorem}[subsection]
\newtheorem{lem}[thm]{Lemma}
\newtheorem{prop}[thm]{Proposition}

\newtheorem{claim}{Claim}[thm]

\newtheorem{ithm}{Theorem}

\newtheorem*{claim*}{Claim}

\theoremstyle{definition}

\theoremstyle{remark}
\newtheorem{rem}[thm]{Remark}

\newenvironment{enu}{%
 \begin{enumerate}%
}{\end{enumerate}}

\begin{document}

\setlength{\baselineskip}{18.3pt}

\title{\Large\bf 
Automorphisms of Niemeier lattices \\[1.5mm]
for Miyamoto's $\BZ_3$-orbifold construction
}
\author{
 Motohiro Ishii%
\footnote{M.I. was partially supported by the Japan Society for the 
Promotion of Science Research Fellowships for Young Scientists, and
by Grant-in-Aid for Research Activity Start-up No. 26887002, Japan} \\
 \small Research Center for Pure and Applied Mathematics, \\
 \small Graduate School of Information Sciences, Tohoku University, \\
 \small Aramaki aza Aoba 6-3-09, Aoba-ku, Sendai 980-8579, Japan \\
 \small (e-mail: {\tt ishii@math.is.tohoku.ac.jp}) \\[5mm]
 Daisuke Sagaki%
\footnote{D.S. was partially supported by 
Grant-in-Aid for Young Scientists (B) No.\,23740003, Japan.} \\
 \small Institute of Mathematics, University of Tsukuba, \\
 \small Tennodai 1-1-1, Tsukuba, Ibaraki 305-8571, Japan \\
 \small (e-mail: {\tt sagaki@math.tsukuba.ac.jp}) \\[5mm]
 Hiroki Shimakura%
\footnote{H.S. was partially supported 
by Grant-in-Aid for Scientific Research (C) No.\,23540013, Japan, 
by Grant-in-Aid for Young Scientists (B) No.\,26800001, Japan, and 
by Grant for Basic Science Research Projects from The Sumitomo Foundation.} \\
 \small Research Center for Pure and Applied Mathematics, \\
 \small Graduate School of Information Sciences, Tohoku University, \\
 \small Aramaki aza Aoba 6-3-09, Aoba-ku, Sendai 980-8579, Japan \\
 \small (e-mail: {\tt shimakura@m.tohoku.ac.jp})
}
\date{}
\maketitle

%
%=======================%
%     START ABSTRACT    %
%=======================%
%

\begin{abstract} \setlength{\baselineskip}{16pt}
We classify, up to conjugation, 
all automorphisms of Niemeier lattices to which
we can apply Miyamoto's orbifold construction. 
Using this classification, we prove that 
the VOAs obtained in \cite{M} and \cite{SS} 
are all of holomorphic non-lattice VOAs 
which we can obtain by applying 
the $\BZ_3$-orbifold construction to 
a Niemeier lattice and its automorphism.
\end{abstract}
%
%=========================%
%     START SECTION 01    %
%=========================%
%
\section{Introduction.}
\label{sec:intro}

% This paper is a continuation of \cite{SS}. 
In \cite{M}, Miyamoto gave a $\BZ_{3}$-orbifold construction 
for holomorphic vertex operator algebras (VOAs for short), 
and obtained a new holomorphic VOA of central charge $24$ 
(whose Lie algebra of 
the weight one subspace is of type $E_{6,3}G_{2,1}^3$) 
by applying his construction to 
the Niemeier lattice $\Ni(E_{6}^{4})$ and 
its automorphism of order $3$ 
(which we denote by $\sigma_{6}$ in \S\ref{subsec:s6}). 
Also, he obtained a holomorphic VOA of central charge $24$ whose weight one subspace is 
identical to $\{0\}$, by applying his $\BZ_{3}$-orbifold 
construction to the Leech lattice VOA and its fixed-point-free 
automorphism of order $3$ (which we denote by $\sigma_{7}$); 
this holomorphic VOA is conjecturally isomorphic to 
the Moonshine VOA $V^{\natural}$. 
Then, in \cite{SS}, 
we found another five pairs of a Niemeier lattice and 
its automorphism of order $3$ from which we can obtain 
new holomorphic VOAs of central charge $24$ by 
Miyamoto's $\BZ_{3}$-orbifold construction 
(for the definitions of 
$\sigma_{1},\,\sigma_{2},\,\dots,\,\sigma_{5}$ in the table below, 
see \S\S\ref{subsec:s1}\,--\,\ref{subsec:s5}):
\begin{equation*}
\begin{array}{c|c|c|c}
\text{Ref.} 
  & \text{\begin{tabular}{c} Niemeier lattice, \\
          Automorphism \end{tabular}}
  & \text{\begin{tabular}{c} Lie algebra structure of \\ 
       the weight one subspace \end{tabular}} 
  & \text{\begin{tabular}{c} 
    No.\,in \\ \cite[Table~1]{Sch} \end{tabular}} \\[1.5mm] 
  \hline\hline
  & & & \\[-4mm]
\text{\cite[\S3]{SS}}
  & \Ni(A_{2}^{12}), \ \sigma_{1} 
  & A_{2,3}^{6} & 6 \\[1mm]
\text{\cite[\S4]{SS}}
  & \Ni(D_{4}^{6}), \ \sigma_{2} 
  & A_{2,3}^{6} & 6 \\[1mm]
\text{\cite[\S5]{SS}}
  & \Ni(D_{4}^{6}), \ \sigma_{3} 
  & E_{6,3}G_{2,1}^{3} & 32 \\[1mm]
\text{\cite[\S6]{SS}}
  & \Ni(D_{4}^{6}),\ \sigma_{4} 
  & A_{5,3}D_{4,3}A_{1,1}^{3} & 17 \\[1mm]
\text{\cite[\S7]{SS}}
  & \Ni(A_{5}^{4}D_{4}), \ \sigma_{5} 
  & A_{5,3}D_{4,3}A_{1,1}^{3} & 17 \\[1mm]
\text{\cite[\S3.2]{M}}
  & \Ni(E_{6}^{4}), \ \sigma_{6}  
  & E_{6,3}G_{2,1}^{3} & 32 \\[1mm]
\text{\cite[\S3.1]{M}}
  & \Lambda, \ \sigma_{7}  
  & \{0\} & 0
\end{array}
\end{equation*}
The purpose of this paper is to prove that 
the VOAs obtained in \cite{M} and \cite{SS} 
are all of the holomorphic non-lattice VOAs 
which we can obtain by this method. Namely, we prove 
that if we apply the $\BZ_{3}$-orbifold construction to 
a Niemeier lattice and its automorphism 
which is not conjugate to any of 
the $\sigma_{1},\,\dots,\,\sigma_{7}$ above, then 
the resulting holomorphic VOA is isomorphic to 
the lattice VOA associated to a Niemeier lattice 
(in fact, if two automorphisms are conjugate to each other, 
then so are the VOAs obtained by the $\BZ_{3}$-orbifold construction; 
see Remark~\ref{rem:SCE}\,(2) below).
For this purpose, we classify, up to conjugation, 
all automorphisms of order $3$ of Niemeier lattices 
to which we can apply the $\BZ_{3}$-orbifold construction.

Let us give an explanation of our result more precisely. 
Given a Niemeier lattice $L$ (i.e., a positive-definite 
even unimodular lattice of rank $24$) and 
its automorphism $\tau \in \Aut L$ of order $3$ 
such that the rank of 
the fixed-point lattice $L^{\tau}$ of $L$ under $\tau$ 
is divisible by $6$ (i.e., $\rank L^{\tau} \in 6\BZ$), 
we can obtain a holomorphic VOA of central charge $24$, 
denoted by $\ti{V}_{L}^{\tau}$ in this paper, by Miyamoto's 
$\BZ_{3}$-orbifold construction (see Theorem~\ref{thm:M}); 
this VOA is a $\BZ_{3}$-graded, simple current extension of 
the fixed-point subVOA $V_{L}^{\tau}$ of the lattice 
VOA $V_{L}$ associated to $L$, under the VOA automorphism of order $3$ 
induced from $\tau \in \Aut L$, which we denote also 
by $\tau \in \Aut V_{L}$.

\begin{ithm}[Theorem~\ref{thm:main4}\,(1)] \label{ithm1}
If $\tau$ is contained in the Weyl group $G_{0}(L)$ for $L$
(see \S\ref{subsec:Niemeier}), 
then the VOA $\ti{V}_{L}^{\tau}$ is isomorphic to 
the lattice VOA associated to a Niemeier lattice. 
\end{ithm}

Thus, for our purpose, we may assume 
that $\tau \notin G_{0}(L)$ (see \eqref{eq:tau}). 
For each $r=0$, $6$, $12$, $18$ (recall that $\rank L^{\tau} \in 6\BZ$), 
denote by $\CC_{r}$ the set of 
conjugacy classes in $\Aut L$ which contain 
elements $\tau \in \Aut L$ satisfying the conditions that 
$|\tau|=3$, $\rank L^\tau=r$, and $\tau \notin G_{0}(L)$. 

\begin{ithm}[Theorem~\ref{thm:mainS3}]  \label{ithm2}
If there exists $\tau \in \Aut L$ 
satisfying the conditions that 
$|\tau|=3$, $\rank L^\tau \in 6\BZ$, 
and $\tau \notin G_{0}(L)$, 
then the root lattice $Q$ of $L$ is isomorphic to one of the following:
\begin{equation*}
\bigl\{ 0 \bigr\} ,\ A_1^{24},\ A_2^{12},\ 
A_3^{8},\ D_4^6,\ A_5^4D_4,\ A_6^4,\ D_6^4,\ E_6^4.
\end{equation*}
For each of these $Q$'s and $r=0$, $6$, $12$, $18$, 
the cardinality $\# \CC_{r}$ of the set $\CC_{r}$ is 
given by the following table. 
\begin{equation*}
\begin{array}{|c||c|c|c|c|c|c|c|c|c|c|} \hline 
 Q & \{0\} &
 A_1^{24} & A_2^{12} & A_3^8 &
 D_4^6 & A_5^4 D_4 & A_6^4 & D_6^4 & E_6^4  \\ \hline\hline
 \#\CC_{0}
  & 1 & 0 & 0 & 0 & 1 & 0 & 0 & 0 & 0  \\ \hline
 \#\CC_{6} 
  & 1 & 0 & 1 & 0 & 2 & 1 & 0 & 0 & 1  \\ \hline
 \#\CC_{12}
  & 1 & 1 & 1 & 1 & 2 & 1 & 2 & 1 & 1  \\ \hline
 \#\CC_{18}
  & 0 & 0 & 0 & 0 & 0 & 0 & 0 & 0 & 0  \\ \hline
\end{array}
\end{equation*}
In particular, there exists no $\tau \in \Aut L$ 
satisfying the conditions that $|\tau|=3$, 
$\rank L^\tau=18$, and $\tau \notin G_{0}(L)$. Also, 
if $Q \ne \{0\}$, and $\rank L^{\tau} \in \bigl\{0,\,6\bigr\}$, 
then $\tau$ is conjugate to one of $\sigma_{1},\,\dots,\,\sigma_{6}$. 
\end{ithm}

In the case that $Q=\{0\}$, i.e., 
$L=\Lambda$ (the Leech lattice), we have the following. 

\begin{ithm}[Theorem~\ref{thm:main4}\,(2)] \label{ithm3}
Assume that $L=\Lambda$. If $\rank \Lambda^{\tau}=0$, then 
$\tau$ is conjugate to $\sigma_{7}$, and hence 
$(\ti{V}_{\Lambda}^{\tau})_{1}=\{0\}$. 
Otherwise, $\ti{V}_{\Lambda}^{\tau} \cong V_{\Lambda}$. 
\end{ithm}

So, let us consider the case that $L \ne \Lambda$. 
If $\rank L^{\tau} =0$ or $6$, then $\tau$ is conjugate to 
one of $\sigma_{1},\,\dots,\,\sigma_{6}$, 
and hence $\ti{V}_{L}^{\tau}$ is isomorphic to one of 
the holomorphic (non-lattice) VOAs obtained in \cite{M} and 
\cite{SS} (see Theorem~\ref{thm:main4}\,(3a)). 
In the case that $\rank L^{\tau} = 12$, we have the following. 

\begin{ithm}[Theorem~\ref{thm:main4}\,(3b)] \label{ithm4}
Let $\tau \in \Aut L$ be such that $|\tau|=3$, 
$\rank L^\tau=12$, and $\tau \notin G_{0}(L)$. 
Then, $\ti{V}_{L}^{\tau} \cong V_{L}$. 
\end{ithm}

This paper is organized as follows: 
In \S\ref{sec:review-Ni}, we review Niemeier lattices and 
their automorphism groups, and then the definitions of 
the automorphisms $\sigma_{1},\,\sigma_{2},\,\dots,\,\sigma_{7}$ 
introduced in \cite{SS} and \cite{M}. 
In \S\ref{sec:Lsigma6}, we prove 
Theorem~\ref{ithm2} above in Theorem~\ref{thm:mainS3}, 
which classifies, up to conjugation, 
all automorphisms of order $3$ 
of Niemeier lattices to which we can apply 
Miyamoto's $\BZ_{3}$-orbifold construction. 
In \S\ref{sec:review}, we briefly review lattice VOAs, 
twisted modules over lattice VOAs, and 
Miyamoto's $\BZ_{3}$-orbifold construction. 
In \S\ref{sec:LieStructure}, 
we prove Theorems~\ref{ithm1}, \ref{ithm3}, 
and \ref{ithm4} above in Theorem~\ref{thm:main4}; 
proofs for parts (1), (2) and (3) of Theorem~\ref{thm:main4}
are given in \S\ref{subsec:weyl}, 
\S\ref{subsec:leech}, and \S\ref{subsec:Nie3}, 
respectively.

\paragraph{Acknowledgments.}
The authors thank Professor Masahiko Miyamoto 
for fruitful comments and discussions.
Also, the authors thank the referee for many variable comments. 

\paragraph{List of Notation.} \quad

\svsp

\begin{center}
{\small
\begin{tabular}{p{25mm}p{130mm}}
$\BF_{n}$ & the field of $n$ elements. \\
$\Aut X$ & the automorphism group of $X$, 
where $X$ is a lattice, a Lie algebra, or a VOA. \\
$\Sym X$ & the symmetric group on a set $X$. \\
$W(Q)$ & the Weyl group of a root lattice $Q$. \\
$\BZ_{n}=\BZ/n\BZ$ & the cyclic group of order $n$. \\
$\FS_{n}$ & the symmetric group of degree $n$. \\
$\FA_{n}$ & the alternating group of degree $n$. \\
$|g|$ & the order of an element $g$ in a group. \\
$X \sbg Y$ & $X$ is a subgroup of $Y$. \\
$X \lhd Y$ & $X$ is a normal subgroup of $Y$. \\
$X:Y$ & a split extension of a group $Y$ by a group $X$. \\
$X.Y$ & an extension of a group $Y$ by a group $X$. \\
$\Con{x}{G}$ & the conjugacy class containing $x$ in a group $G$. \\
$L^{\ast}$ & the dual lattice of a lattice $L$. \\
$L^{\tau}$ & the fixed-point sublattice of a lattice $L$ under $\tau \in \Aut L$. \\
$\Ni(Q)$ & the Niemeier lattice whose root lattice is $Q$. \\
$\SC_Q$ & 
the set of indecomposable components of a root lattice $Q$. \\
$\Lambda$ & the Leech lattice. \\
$V_{L}$ & the lattice vertex operator algebra 
associated to a lattice $L$. \\
$\Fg(X)$ & the semisimple Lie algebra of type $X$. 
\end{tabular}
} 
\end{center}
%
%=========================%
%     START SECTION 02    %
%=========================%
%
\section{Review.}
\label{sec:review-Ni}
In this section, we review Niemeier lattices and their automorphism groups in 
\S\ref{subsec:Niemeier}, and then the definition of 
the automorphisms $\sigma_{1},\,\sigma_{2},\,\dots,\,\sigma_{7}$ 
introduced in \cite{SS} and \cite{M}. 
%
%==============================%
%     START SUBSECTION 0201    %
%==============================%
%
\subsection{Niemeier lattices and their automorphism groups.}
\label{subsec:Niemeier}

A Niemeier lattice is, by definition, 
a positive-definite even unimodular lattice of rank $24$;
for the classification of Niemeier lattices, 
see \cite[Table 16.1]{CS}. 
For a Niemeier lattice $L$ with $\BZ$-bilinear form $\pair{\cdot\,}{\cdot}$, 
we define its root lattice $Q$ to be the sublattice generated by 
$\Delta:=\bigl\{ \alpha \in L \mid \pair{\alpha}{\alpha}=2 \bigr\}$. 
Recall that every Niemeier lattice is uniquely (up to an isomorphism) 
determined by its root lattice; denote by $\Ni(Q)$ the Niemeier lattice 
whose root lattice is $Q$. If $Q=\{0\}$, then $\Ni(Q)$ is isomorphic 
to the Leech lattice $\Lambda$. 
Assume that $Q \ne \{0\}$. Then it is known that $\rank Q=24$, and 
$\Ni(Q)$ can be realized as a sublattice of the dual lattice $Q^{\ast}$ of $Q$. 
Thus, $\Ni(Q)/Q$ is a finite abelian group, 
which we call the glue code or the (set of) glue vectors.

Now, let $L=\Ni(Q)$ be a Niemeier lattice such that $Q \ne \{0\}$. 
First we review the group structure of 
the automorphism group $\Aut Q$ of the root lattice $Q$. 
Let $Q=\bigoplus_{m=1}^n Q_m$ be 
the decomposition of $Q$ into its indecomposable components; 
we know from \cite[Corollary~5.10\,b)]{Kac} that 
for each $1 \le m \le n$, 
\begin{equation*}
\Aut Q_{m}=W(Q_{m}) : G_{1}(Q_{m}),
\end{equation*}
where $W(Q_{m}) \sbg \Aut Q_{m}$ is 
the Weyl group of $Q_{m}$, and $G_{1}(Q_{m})$ 
is the subgroup of $\Aut Q_{m}$ consisting of 
all Dynkin diagram automorphisms of $Q_{m}$ 
(with respect to a fixed set $\Pi_{m}$ of simple roots of $Q_{m}$). 
Here we set
%
%%%%%%%%%%%%%%%
%%% eq:G01Q %%%
%%%%%%%%%%%%%%%
%
\begin{equation} \label{eq:G01Q}
G_{0}(Q) : = \prod _{m=1} ^n W(Q_m), \qquad 
G_{1}(Q) : = \prod _{m=1} ^n G_{1}(Q_m),
\end{equation}
%
%%%%%%%%%%%%%
%%% eq:KQ %%%
%%%%%%%%%%%%%
%
\begin{equation} \label{eq:KQ}
K(Q):=\bigl\{\tau \in \Aut Q \mid \tau(Q_{m})=Q_{m} \ 
\text{for all $1 \le m \le n$}\bigr\} \sbg \Aut Q; 
\end{equation}
we call $G_{0}(Q)$ the Weyl group of $Q$. 
Remark that $G_{0}(Q) \lhd K(Q)$, 
$G_{1}(Q) \sbg K(Q)$, and 
%
%%%%%%%%%%%%%%%%%
%%% eq:KQ-G01 %%%
%%%%%%%%%%%%%%%%%
%
\begin{equation} \label{eq:KQ-G01}
K(Q)=\prod_{m=1}^{n} \Aut Q_{m}
    =G_{0}(Q):G_{1}(Q).
\end{equation}
For each $1 \le i < j \le n$ such that $Q_{i} \cong Q_{j}$, 
we have the following automorphism $t_{ij} \in \Aut Q$ 
of $Q=\bigoplus_{m=1}^{n} Q_{m}$ 
(the ``transposition'' of the $i$-th entry 
and the $j$-th entry):
\begin{equation*}
(x_{1},\dots,\,x_{i},\,\dots,\,x_{j},\,\dots,\,x_{n}) 
\mapsto 
(x_{1},\dots,\,x_{j},\,\dots,\,x_{i},\,\dots,\,x_{n}).
\end{equation*}
We set
\begin{equation*}
G_{2}(Q):=\bigl\langle 
t_{ij} \mid 1 \le i < j \le n \text{ such that } 
Q_{i} \cong Q_{j} \bigr\rangle \sbg \Aut Q,
\end{equation*}
which is the subgroup of $\Aut Q$ consisting of 
all ``permutations'' of entries of $Q=\bigoplus_{m=1}^{n} Q_{m}$. 
Then it can be easily verified that 
%
%%%%%%%%%%%%%%%
%%% eq:AutQ %%%
%%%%%%%%%%%%%%%
%
\begin{equation} \label{eq:AutQ}
\Aut Q = K(Q):G_{2}(Q) = G_{0}(Q):G_{1}(Q):G_{2}(Q).
\end{equation}
We see that $G_{1}(Q):G_{2}(Q)$ is the subgroup of $\Aut Q$ 
consisting of all elements in $\Aut Q$ that preserves 
the set $\Pi:=\bigsqcup_{m=1}^{n} \Pi_{m}$ of simple roots of $Q$. 

%%%%%%%%%%%%%%
%%% rem:G2 %%%
%%%%%%%%%%%%%%
%
\begin{rem} \label{rem:G2}
Notice that $\Aut Q$ naturally acts on 
the set $\SC_Q := \bigl\{ Q_1 , \ldots , Q_n \bigr\}$ 
of indecomposable components of $Q$. 
Hence we have a group homomorphism 
$\Phi : \Aut Q \rightarrow \Sym \SC_Q$,
where $\Sym \SC_{Q}$ is the symmetric group 
on the set $\SC_{Q}$. It is obvious that 
\begin{equation*}
G_{2}(Q) \cong \Img \Phi, \qquad 
\Ker \Phi = K(Q)=G_{0}(Q):G_{1}(Q).
\end{equation*}
\end{rem}

Next, let us review from \cite[\S3 in Chapter 4]{CS}
the group structure of the automorphism group $\Aut L$ of $L$. 
Notice that $\Aut L \sbg \Aut Q$. 
Indeed, since the spanning set 
$\Delta=\bigl\{ \alpha \in L \mid \pair{\alpha}{\alpha} =2 \bigr\}$ 
of $Q$ is stable under the action of $\Aut L$, 
it follows immediately that $Q$ is stable under $\Aut L$. 
Thus we get the natural group homomorphism 
$\Aut L \rightarrow \Aut Q$ defined 
by the restriction $\tau \mapsto \tau|_{Q}$ 
for $\tau \in \Aut L$. 
Since $L \otimes_{\BZ} \BR = Q \otimes _{\BZ} \BR$, 
we see that this homomorphism is injective. 

It is well-known (and easily verified) that the Weyl group 
$G_{0}(Q)=\prod _{m=1} ^n W(Q_m)$ is contained in $\Aut L$. 
Set
\begin{equation*}
G_{0}(L):=G_{0}(Q) \lhd \Aut L, 
\end{equation*}
\begin{equation*}
G_{1}(L):=\Aut L \cap G_{1}(Q)
 \sbg \Aut L;
\end{equation*}
note that $G_{0}(L)$ and $G_{1}(L)$ are contained in 
$K(Q) \cap \Aut L$ (see \eqref{eq:KQ} and \eqref{eq:KQ-G01}), 
and 
%
%%%%%%%%%%%%%%%
%%% eq:KQ-L %%%
%%%%%%%%%%%%%%%
%
\begin{equation} \label{eq:KQ-L}
K(Q) \cap \Aut L = G_{0}(L) : G_{1}(L).
\end{equation}
Furthermore we can easily show that 
%
%%%%%%%%%%%%%
%%% eq:HL %%%
%%%%%%%%%%%%%
%
\begin{equation} \label{eq:HL}
\begin{split}
& \Aut L = G_{0}(L):H(L), \quad \text{where} \\
& H(L):=\Aut L \cap \bigl(G_{1}(Q):G_{2}(Q)\bigr) \sbg \Aut L;
\end{split}
\end{equation}
for each $\tau \in H(L)$, 
there exist unique $\tau_{1} \in G_{1}(Q)$ and 
$\tau_{2} \in G_{2}(Q)$ such that $\tau=\tau_{1}\tau_{2}$, 
but we should remark that $\tau_{1}$ and $\tau_{2}$ 
are not necessarily contained in $\Aut L$. Define 
%
%%%%%%%%%%%%%%
%%% eq:G2L %%%
%%%%%%%%%%%%%%
%
\begin{equation} \label{eq:G2L}
G_{2}(L) : = 
 \bigl\{ \tau \in G_{2}(Q) \mid 
 \text{$\tau_1 \tau \in H(L)$ for some $\tau_1 \in G_{1}(Q)$} \bigr\}. 
\end{equation}
Because $G_{1}(Q) \lhd G_{1}(Q):G_{2}(Q)$, 
we deduce that $G_{2}(L)$ is a subgroup of $G_{2}(Q)$. 
In addition, if $\tau=\tau_{1}\tau_{2} \in H(L)$ 
with $\tau_{1} \in G_{1}(Q)$ and $\tau_{2} \in G_{2}(Q)$, 
then it is obvious by the definition that $\tau_{2} \in G_{2}(L)$. 
Thus we obtain a map $\pi_{2}:H(L) \rightarrow G_{2}(L)$, 
$\tau \mapsto \tau_{2}$, which is obviously surjective. 
Also, it can be verified that $\pi_{2}$ is 
a group homomorphism, with $G_{1}(L)=\Aut L \cap G_{1}(Q)$ 
as the kernel. Thus we obtain the following exact sequence:
%
%%%%%%%%%%%%%%%%
%%% eq:exact %%%
%%%%%%%%%%%%%%%%
%
\begin{equation} \label{eq:exact}
1 \longrightarrow G_{1}(L) 
\stackrel{\subset}{\longrightarrow} H(L) 
\stackrel{\pi_{2}}{\longrightarrow} G_{2}(L) 
\longrightarrow 1.
\end{equation}
%
%%%%%%%%%%%%%%%
%%% rem:G2L %%%
%%%%%%%%%%%%%%%
%
\begin{rem} \label{rem:G2L}
With the notation in Remark~\ref{rem:G2}, 
we have $G_{2}(L) \cong \Phi ( \Aut L ) \sbg \Sym \SC_{Q}$. 
\end{rem}

\begin{rem}
Because $\Aut L \sbg \Aut Q$, we have the induced action of $\Aut L$ 
on the glue code $L/Q$. Then, $\tau \in \Aut L$ acts on $L/Q$
as the identity map if and only if $\tau \in G_{0}(L)$. 
Hence, by \eqref{eq:HL}, $H(L)$ is identical to 
the subgroup of $G_{1}(Q):G_{2}(Q)$ consisting of the elements 
that preserves the glue code $L/Q$.
\end{rem}

We know from \cite[\S1 in Chapter~16 and \S4 in Chapter~18]{CS} 
the group structures of $G_{1}(L)$ and $G_{2}(L)$ 
for each Niemeier lattice $L$ whose root lattice $Q$ is 
one of those in \eqref{eq:Q} below, except for $Q=\{0\}$:  
%
%%%%%%%%%%%%%%%%%%
%%% table:G1G2 %%%
%%%%%%%%%%%%%%%%%%
%
\begin{equation} \label{table:G1G2}
\begin{array}{|c||c|c|c|c|c|c|c|c|c|c|} \hline 
 Q & A_1^{24} & A_2^{12} & A_3^8 & D_4^6 
   & A_5^4D_4 & A_6^4 & D_6^4 & E_6^4 \\ \hline \hline
 G_1 (L) & 1 & \BZ_2 & \BZ_2 & \BZ_3 & 
 \BZ_2 & \BZ_2 & 1 & \BZ_2 \\\hline
 G_2 (L) & M_{24} & M_{12} & AGL_{3}(2) & \FS_6 & 
 \FS_4 & \FA_4 & \FS_4 & \FS_4 \\\hline
\end{array}
\end{equation}
Here, $M_{24}$ and $M_{12}$ are the Mathieu groups of degree $24$ and $12$, respectively, 
and $AGL_{3}(2)$ is the affine general linear group of degree $3$ over $\BF_{2}$. 

If $Q=\{0\}$, or equivalently, $L \cong \Lambda$ (the Leech lattice), 
then $\Aut \Lambda$ is isomorphic to a (unique) nontrivial extension of 
the largest Conway's sporadic simple group $\Co_{1}$ by $\BZ_{2}$.
Let $\sigma_{7} \in \Aut \Lambda$ be a fixed-point-free automorphism of order $3$; 
we will show at the beginning of \S\ref{subsec:main32} below that 
such an automorphism of $\Lambda$ exists uniquely, up to conjugation. 

%=================================%
%     START SUBSECTION 0202       %
%=================================%
%
\subsection{Root lattices.}
\label{subsec:root}

In this subsection, we fix the notation for some root lattices, 
which are needed to define the automorphisms 
$\sigma_{1},\,\sigma_{2},\,\dots,\,\sigma_{6}$ 
introduced in \cite{SS} and \cite{M}. 
Also, we introduce three automorphisms $\omega$, $\vp$, and $\psi$ 
of the root lattice $D_{4}$ for later use. 
%
%%%%%%%%%%%%%%%%%%%%%%%%%%%%%%%%%
\paragraph{Root lattice $A_{n}$.}
%%%%%%%%%%%%%%%%%%%%%%%%%%%%%%%%%
%
Following \cite[Chapter 4, \S6.1]{CS}, we set
\begin{equation*}
A_{n}:=\bigl\{ 
(x_{0},\,x_{1},\,\dots,\,x_{n}) \in \BZ^{n+1} 
 \mid x_{0}+x_{1}+ \cdots + x_{n} =0
\bigr\},
\end{equation*}
\begin{equation*}
[\ell] : = 
\frac{1}{n+1}(\ell,\,\dots,\,\ell,\,
\underbrace{\ell-n-1,\,\dots,\,\,\ell-n-1}_{%
 \text{$\ell$ times}}) \in A_{n}^{\ast}
\qquad \text{for $\ell=0,\,1,\,\dots,\,n$}. 
\end{equation*}
Then, $A_{n}^{\ast}/A_{n}=\bigl\{\ol{[\ell]} : = [\ell]+A_{n} \mid 
\ell=0,\,1,\,\dots,\,n\bigr\} \cong \BZ_{n+1}$. 
Recall that $\Aut A_{1} \cong \BZ_{2}$, and 
$\Aut A_{n} \cong \BZ_{2} \times \FS_{n+1}$ for $n \ge 2$.

%%%%%%%%%%%%%%%%%%%%%%%%%%%%%%%%%
\paragraph{Root lattice $D_{n}$.}
%%%%%%%%%%%%%%%%%%%%%%%%%%%%%%%%%
%
Following \cite[Chapter~4, \S7.1]{CS}, 
we set
\begin{equation*}
D_{n}:=\bigl\{(x_{1},\,x_{2},\,\dots,\,x_{n}) \in \BZ^{n} \mid 
x_{1}+x_{2}+\cdots+x_{n} \in 2\BZ\bigr\}, 
\end{equation*}
\begin{align*}
& [0]:=(0,\,0,\,\dots,\,0) \in D_{n}^{\ast}, & 
& [1]:=(1/2,\,1/2,\,\dots,\,1/2) \in D_{n}^{\ast}, \\
& [2]:=(0,\,\dots,\,0,\,1) \in D_{n}^{\ast}, & 
& [3]:=(1/2,\,\dots,\,1/2,\,-1/2) \in D_{n}^{\ast}. 
\end{align*}
Then we have $D_{n}^{\ast}/D_{n}=
\bigl\{\ol{[\ell]}:=[\ell]+D_{n} \mid \ell=0,\,1,\,2,\,3\bigr\}$; 
in particular, $D_{4}^{\ast}/D_{4} \cong \BZ_{2} \times \BZ_{2}$. 
Recall that $\Aut D_{n} \cong \BZ_{2}^{n} : \FS_{n}$ for $n \ge 5$, and 
$\Aut D_{4} \cong \BZ_{2}^{3} : \FS_4 : \FS_{3}$. 

In the case of $D_{4}$, 
\begin{align*}
 & \alpha_{1} := (1,\,-1,\,0,\,0), &
 & \alpha_{2} := (0,\,1,\,-1,\,0), \\
 & \alpha_{3} := (0,\,0,\,1,\,-1), &
 & \alpha_{4} := (0,\,0,\,1,\,1)
\end{align*}
give a set of simple roots such that
$\pair{\alpha_{2}}{\alpha_{i}}=-1$ for $i=1,\,3,\,4$, and 
$\pair{\alpha_{k}}{\alpha_{l}}=0$ for $k,\,l \in \{1,\,3,\,4\}$, $k \ne l$. 
Let $\omega$ be the Dynkin diagram automorphism of $D_{4}$ of order $3$ such that 
$\omega(\alpha_{1})=\alpha_{3}$, 
$\omega(\alpha_{3})=\alpha_{4}$, 
$\omega(\alpha_{4})=\alpha_{1}$, and 
$\omega(\alpha_{2})=\alpha_{2}$; 
we can easily check that 
\begin{equation} \label{eq:omega-glue}
\omega(\ol{[0]})=\ol{[0]}, \quad
\omega(\ol{[1]})=\ol{[2]}, \quad
\omega(\ol{[2]})=\ol{[3]}, \quad
\omega(\ol{[3]})=\ol{[1]}.
\end{equation}
Also, we define a linear automorphism $\vp$ of 
$D_{4} \otimes_{\BZ} \BR$ by: 
\begin{align*}
& (1,0,0,0) \mapsto \frac{1}{2}(-1,\,1,\,1,\,1), &
& (0,1,0,0) \mapsto \frac{1}{2}(-1,\,-1,\,1,\,-1), \\[1mm]
& (0,0,1,0) \mapsto \frac{1}{2}(-1,\,-1,\,-1,\,1), & 
& (0,0,0,1) \mapsto \frac{1}{2}(-1,\,1,\,-1,\,-1);
\end{align*}
we see that the restriction of $\vp$ to $D_{4}$ is 
a lattice automorphism of order $3$ of $D_{4}$ 
which is fixed point-free on $D_{4}$. Observe that 
\begin{equation} \label{eq:vpi-glue}
\vp(\ol{[0]})=\ol{[0]}, \quad
\vp(\ol{[1]})=\ol{[2]}, \quad
\vp(\ol{[2]})=\ol{[3]}, \quad
\vp(\ol{[3]})=\ol{[1]}.
\end{equation}
Hence, by \eqref{eq:omega-glue} and \eqref{eq:vpi-glue}, 
the action of $\vp$ on $D_{4}^{\ast}/D_{4}$ is 
identical to that of $\omega$, 
which implies that $\vp \in W(D_{4})\omega$. 
Also, define $\psi:=r_{1}r_{2} \in W(D_{4})$, 
where $r_{i}$ denotes the simple reflection with respect to 
the simple root $\alpha_{i}$ for $i=1,\,2,\,3,\,4$; note that 
$\psi$ is of order $3$. We have
%
%%%%%%%%%%%%%%%%%%%
%%% eq:D4-fixed %%%
%%%%%%%%%%%%%%%%%%%
%
\begin{equation} \label{eq:D4-fixed}
\rank D_4^\omega=2, \quad 
\rank D_4^\vp=0, \quad 
\rank D_4^\psi=2.
\end{equation}

%%%%%%%%%%%%%%%%%%%%%%%%%%%%%%%%%
\paragraph{Root lattice $E_{6}$.}
%%%%%%%%%%%%%%%%%%%%%%%%%%%%%%%%%
%
Let $\bigl\{\alpha_i\mid 1\le i\le 6\bigr\}$ be 
the set of simple roots for the root lattice 
$E_6=\bigoplus_{i=1}^{6}\BZ\alpha_{i}$; 
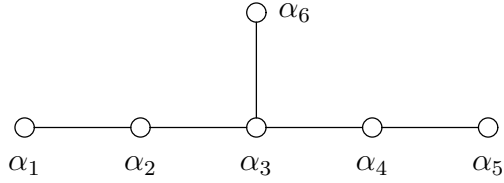
\begin{center}
%WinTpicVersion3.08
\unitlength 0.1in
\begin{picture}( 29.0000,  8.0000)( -0.5000, -9.1500)
% CIRCLE 2 0 3 0
% 4 400 800 400 750 400 750 400 750
% 
\special{pn 8}%
\special{ar 400 800 50 50  0.0000000 6.2831853}%
% CIRCLE 2 0 3 0
% 4 1000 800 1000 750 1000 750 1000 750
% 
\special{pn 8}%
\special{ar 1000 800 50 50  0.0000000 6.2831853}%
% CIRCLE 2 0 3 0
% 4 1600 800 1600 750 1600 750 1600 750
% 
\special{pn 8}%
\special{ar 1600 800 50 50  0.0000000 6.2831853}%
% CIRCLE 2 0 3 0
% 4 2200 800 2200 750 2200 750 2200 750
% 
\special{pn 8}%
\special{ar 2200 800 50 50  0.0000000 6.2831853}%
% CIRCLE 2 0 3 0
% 4 2800 800 2800 750 2800 750 2800 750
% 
\special{pn 8}%
\special{ar 2800 800 50 50  0.0000000 6.2831853}%
% CIRCLE 2 0 3 0
% 4 1600 200 1600 150 1600 150 1600 150
% 
\special{pn 8}%
\special{ar 1600 200 50 50  0.0000000 6.2831853}%
% LINE 2 0 3 0
% 2 450 800 950 800
% 
\special{pn 8}%
\special{pa 450 800}%
\special{pa 950 800}%
\special{fp}%
% LINE 2 0 3 0
% 2 1050 800 1550 800
% 
\special{pn 8}%
\special{pa 1050 800}%
\special{pa 1550 800}%
\special{fp}%
% LINE 2 0 3 0
% 2 1650 800 2150 800
% 
\special{pn 8}%
\special{pa 1650 800}%
\special{pa 2150 800}%
\special{fp}%
% LINE 2 0 3 0
% 2 2250 800 2750 800
% 
\special{pn 8}%
\special{pa 2250 800}%
\special{pa 2750 800}%
\special{fp}%
% LINE 2 0 3 0
% 2 1600 750 1600 250
% 
\special{pn 8}%
\special{pa 1600 750}%
\special{pa 1600 250}%
\special{fp}%
% STR 2 0 3 0
% 3 400 900 400 1000 5 0
% $\alpha_1$
\put(4.0000,-10.0000){\makebox(0,0){$\alpha_1$}}%
% STR 2 0 3 0
% 3 1000 900 1000 1000 5 0
% $\alpha_2$
\put(10.0000,-10.0000){\makebox(0,0){$\alpha_2$}}%
% STR 2 0 3 0
% 3 1600 900 1600 1000 5 0
% $\alpha_3$
\put(16.0000,-10.0000){\makebox(0,0){$\alpha_3$}}%
% STR 2 0 3 0
% 3 2200 900 2200 1000 5 0
% $\alpha_4$
\put(22.0000,-10.0000){\makebox(0,0){$\alpha_4$}}%
% STR 2 0 3 0
% 3 2800 900 2800 1000 5 0
% $\alpha_5$
\put(28.0000,-10.0000){\makebox(0,0){$\alpha_5$}}%
% STR 2 0 3 0
% 3 1800 100 1800 200 5 0
% $\alpha_6$
\put(18.0000,-2.0000){\makebox(0,0){$\alpha_6$}}%
\end{picture}%
\end{center}
\noindent
We set 
\begin{equation*}
[0]:=0,\quad [1]:=\frac{1}{3}\left(\alpha_1-\alpha_2+\alpha_4-\alpha_5\right),\quad 
[2]:=\frac{1}{3}\left(-\alpha_1+\alpha_2-\alpha_4+\alpha_5\right). 
\end{equation*}
Then we have $E_{6}^{\ast}/E_{6}=\bigl\{\ol{[\ell]} := [\ell]+E_{6} 
\mid \ell=0,\,1,\,2\bigr\} \cong \BZ/3\BZ$. 
% Recall that $\Aut E_6 \cong \BZ_2 \times U_4 (2).\BZ_2$ (see \cite[$U_4(2)$]{ATLAS}). 

%=================================%
%     START SUBSECTION 0203       %
%=================================%
%
\subsection{Niemeier lattice $\Ni(A_{2}^{12})$ and 
its automorphism $\sigma_{1}$ of order $3$.}
\label{subsec:s1}
Define $Q$ to be the direct sum $A_{2}^{12}$ of 
12 copies of the root lattice $A_{2}$; 
following \cite[Chapter 10, \S1.5]{CS}, 
we use $\Omega:=\bigl\{\infty,\,0,\,1,\,\dots,\,10\bigr\}$ 
for the index set of the coordinate for $Q$, that is, 
$Q=\bigl\{(\alpha_{i})_{i \in \Omega} \mid 
  \text{$\alpha_{i} \in A_{2}$ for $i \in \Omega$}
  \bigr\}$.
We can identify $Q^{\ast}/Q$ with the $12$-dimensional 
vector space $\BF_{3}^{12}$ over the field $\BF_{3}$ of three elements. 
For each $j \in \Omega$, define $\ol{[1]}^{(j)}$ to be the element 
$(\ol{[\ell_{i}]})_{i \in \Omega} \in Q^{\ast}/Q$ 
with $\ell_{j}=1$ and $\ell_{i}=0$ 
for all $i \in \Omega$, $i \ne j$. 
Then, $\bigl\{\ol{[1]}^{(j)} \mid j \in \Omega\bigr\}$ 
forms a basis of $Q^{\ast}/Q$. 

Define $\nu \in \Sym \Omega$ by: 
\begin{equation*}
\nu=(\infty)(\ten 9876543210), 
\end{equation*}
where $\ten$ denotes $10$. 
Set $\Theta:=\bigl\{0,\,1,\,3,\,4,\,5,\,9\bigr\} \subset \Omega$, 
and define 
\begin{equation*}
w_{0}:=\sum_{i \in \Omega \setminus \Theta}\ol{[1]}^{(i)}- 
\sum_{j \in \Theta}\ol{[1]}^{(j)}, 
\end{equation*}
\begin{equation*}
w_{i}:=\nu^{i} \cdot w_{0} \quad \text{for $0 \le i \le 10$}, 
\qquad
w_{\infty}:=\sum_{i \in \Omega}\ol{[1]}^{(i)},
\end{equation*}
where the group $\Sym \Omega$ acts linearly on $Q^{\ast}/Q$ by: 
$g \cdot \ol{[1]}^{(i)} = \ol{[1]}^{(g(i))}$ 
for $g \in \Sym \Omega$ and $i \in \Omega$. 
The glue code $\Ni(A_{2}^{12})/Q$ of the Niemeier lattice $\Ni(A_{2}^{12})$ 
is identical to the subspace $\MCD_{12}$ of $Q^{\ast}/Q$ 
spanned by $\bigl\{w_{i} \mid i \in \Omega\bigr\}$
(see \cite[Chapter 18, \S4, I\hspace{-1pt}I]{CS}), that is,
\begin{equation*}
(Q \subset) \quad 
\Ni(A_{2}^{12})=\bigsqcup_{C \in \MCD_{12}} C \quad (\subset Q^{\ast}). 
\end{equation*}
Define $\sigma'=(\infty)(4)(7)(012)(35\ten)(689) \in \Sym \Omega$; 
we see from \cite[Chapter 10, Theorems 2 and 3]{CS} 
(see also \cite[Theorem~3.2.1]{SS}) that 
$\MCD_{12}$ is stable under the action of $\sigma'$. 

We arrange the coordinate of 
$Q=A_{2}^{12}$ as follows:
\begin{equation*}
(\mu_{i})_{i \in \Omega}=
\bigl(
 \underbrace{\mu_{\infty},\,\mu_{4},\,\mu_{7}}_{\in A_{2}^{3}} \mid
 \underbrace{\mu_{0},\,\mu_{3},\,\mu_{6}}_{\in A_{2}^{3}} \mid
 \underbrace{\mu_{1},\,\mu_{5},\,\mu_{8}}_{\in A_{2}^{3}} \mid
 \underbrace{\mu_{2},\,\mu_{10},\,\mu_{9}}_{\in A_{2}^{3}}
\bigr).
\end{equation*}
Let $\psi_{1}$ be the fixed-point-free automorphism of $A_{2}$ 
defined by: $(x_0,x_1,x_2) \mapsto (x_2,x_0,x_1)$; note that 
$\psi_{1}(\ol{[\ell]})=\ol{[\ell]}$ for every $\ell=0,\,1,\,2$. 
Define $\sigma_{1}:Q^{\ast} \rightarrow Q^{\ast}$ by: 
\begin{equation*}
\bigl(
 \mu_{\infty},\,\mu_{4},\,\mu_{7} \mid \bmu_{036} \mid 
 \bmu_{158} \mid \bmu_{2\ten9}
\bigr) 
\stackrel{\sigma_{1}}{\longrightarrow} 
\bigl(
 \psi_{1}(\mu_{\infty}),\,\psi_{1}(\mu_{4}),\,\psi_{1}(\mu_{7}) \mid 
  \bmu_{2\ten9} \mid \bmu_{036} \mid \bmu_{158}
\bigr)
\end{equation*}
for $\mu_{\infty},\,\mu_{4},\,\mu_{7} \in A_{2}^{\ast}$ and 
$\bmu_{036},\,\bmu_{158},\,\bmu_{2\ten9} \in (A_{2}^{\ast})^{3}$. 
By the argument above, we see that $\sigma_{1}$ stabilizes the Niemeier 
lattice $\Ni(A_{2}^{12})$, and hence $\sigma_{1} \in \Aut \Ni(A_{2}^{12})$. 
It can be easily seen that
\begin{equation} \label{eq:s1}
\rank \Ni(A_{2}^{12})^{\sigma_1} = 6.
\end{equation}
%
%=================================%
%     START SUBSECTION 0204       %
%=================================%
%
\subsection{Niemeier lattice $\Ni(D_{4}^{6})$ and 
its automorphisms $\sigma_{2}$, $\sigma_{3}$, $\sigma_{4}$ of order $3$.}
\label{subsec:s234}

By \cite[Chapter 16, Table 16.1]{CS}, 
the glue code $\Ni(D_{4}^{6})/Q$ of the Niemeier lattice $\Ni(D_{4}^{6})$ is 
generated by the cosets in $Q^{\ast}/Q$ containing $[111111]$, $[222222]$, and 
\begin{equation*}
[002332],\ [023320],\ [033202],\ [032023],\ [020233], 
\end{equation*}
where $[a_{1} \cdots a_{6}]:=([a_{1}],\,\dots,\,[a_{6}]) \in 
Q^{\ast}=(D_{4}^{\ast})^{6}$. Namely, $\Ni(D_{4}^{6})$ is 
the sublattice of $Q^{\ast}$ generated by $Q$ and these $7$ elements in $Q^{\ast}$. 

Let us define $\sigma_{2},\,\sigma_{3},\,\sigma_{4}:
Q^{\ast} \rightarrow Q^{\ast}$ by: 
\begin{align*}
& \sigma_{2}(\gamma_{1},\,\dots,\,\gamma_{6})=
(\vp(\gamma_{1}),\,\dots,\,\vp(\gamma_{6})), \\
& \sigma_{3}(
 \gamma_{1},\,\gamma_{2},\,\gamma_{3},\,
 \gamma_{4},\,\gamma_{5},\,\gamma_{6})=
(
 \vp(\gamma_{1}),\,\vp(\gamma_{2}),\,\vp(\gamma_{3}),\,
 \omega(\gamma_{4}),\,\omega(\gamma_{5}),\,\omega(\gamma_{6})), \\
& \sigma_{4}
(\gamma_{1},\,\gamma_{2},\,\gamma_{3},\,\gamma_{4},\,\gamma_{5},\,\gamma_{6})=
(\psi(\gamma_{1}),\,\vp(\gamma_{2}),\,\vp^{-1}(\gamma_{3}),\,
 \gamma_{6},\,\vp^{-1}(\gamma_{4}),\,\vp(\gamma_{5})). 
\end{align*}
We know from \cite[\S4.2, \S5.2, \S6.2]{SS} that 
$\Ni(D_{4}^{6}) \subset Q^{\ast}$ is stable under 
the actions of $\sigma_{2}$, $\sigma_{3}$, $\sigma_{4}$, which implies that 
$\sigma_{2},\,\sigma_{3},\,\sigma_{4} \in \Aut \Ni(D_{4}^{6})$, and also that 
%
%%%%%%%%%%%%%%%
%%% eq:s234 %%%
%%%%%%%%%%%%%%%
%
\begin{equation} \label{eq:s234}
\rank \Ni(D_{4}^{6})^{\sigma_2} = 0, \quad 
\rank \Ni(D_{4}^{6})^{\sigma_3} = 6, \quad 
\rank \Ni(D_{4}^{6})^{\sigma_4} = 6. 
\end{equation}
%
%=================================%
%     START SUBSECTION 0204       %
%=================================%
%
\subsection{Niemeier lattice $\Ni(A_{5}^{4}D_{4})$ and 
its automorphism $\sigma_{5}$ of order $3$.}
\label{subsec:s5}

By \cite[Chapter 16, Table 16.1]{CS}, the glue code $\Ni(A_{5}^{4}D_{4})/Q$ of 
the Niemeier lattice $\Ni(A_{5}^{4}D_{4})$ is generated by the cosets in $Q^{\ast}/Q$ 
containing $[33001]$, $[30302]$, $[30033]$, and 
$[20240]$, $[22400]$, $[24020]$, 
where $[a_{1} \cdots a_{5}]:=([a_{1}],\,\dots,\,[a_{5}]) \in 
Q^{\ast}=(A_{5}^{\ast})^{4}D_{4}^{\ast}$. 
Namely, $\Ni(A_{5}^{4}D_{4})$ is 
the sublattice of $Q^{\ast}$ generated by $Q$ and these $6$ elements in $Q^{\ast}$. 

Let us define $\sigma_{5}:Q^{\ast} \rightarrow Q^{\ast}$ by: 
\begin{equation*}
\sigma_{5}(\gamma_{1},\,\gamma_{2},\,\gamma_{3},\,\gamma_{4},\,\gamma_{5})=
(\psi_{2}(\gamma_{1}),\,\gamma_{4},\,\gamma_{2},\,\gamma_{3},\,\vp(\gamma_{5})),
\end{equation*}
where $\psi_{2}$ is the automorphism of $A_{5}$ defined by: 
\begin{equation*}
(x_{0},\,x_{1},\,x_{2},\,x_{3},\,x_{4},\,x_{5})
 \mapsto
(x_{2},\,x_{0},\,x_{1},\,x_{5},\,x_{3},\,x_{4}).
\end{equation*}
We know from \cite[\S7.2]{SS} that 
$\Ni(A_{5}^{4}D_{4}) \subset Q^{\ast}$ is stable under 
the action of $\sigma_{5}$, which implies that 
$\sigma_{5} \in \Aut \Ni(A_{5}^{4}D_{4})$, and also that 
%
%%%%%%%%%%%%%
%%% eq:s5 %%%
%%%%%%%%%%%%%
%
\begin{equation} \label{eq:s5}
\rank \Ni(A_{5}^{4}D_{4})^{\sigma_{5}} = 6.
\end{equation}
%
%=================================%
%     START SUBSECTION 0205       %
%=================================%
%
\subsection{Niemeier lattice $\Ni(E_{6}^{4})$ and 
its automorphism $\sigma_{6}$ of order $3$.}
\label{subsec:s6}

By \cite[Chapter 16, Table 16.1]{CS}, the glue code $\Ni(E_{6}^{4})/Q$ 
of the Niemeier lattice $\Ni(E_{6}^{4})$ is generated by the cosets in $Q^{\ast}/Q$ 
containing $[1012]$, $[1120]$, $[1201]$, 
where $[a_{1} \cdots a_{4}]:=([a_{1}],\,\dots,\,[a_{4}]) \in 
Q^{\ast}=(E_{6}^{\ast})^{4}$. Namely, $\Ni(E_{6}^{4})$ is 
the sublattice of $Q^{\ast}$ generated by $Q$ and 
$[1012]$, $[1120]$, $[1201] \in Q^{\ast}$. 

Define $\psi_{3}:=r_1r_2r_4r_5r_6r_0 \in W(E_{6})$, where 
$r_i$ denotes the simple reflection with respect to $\alpha_i$ for $1 \le i \le 6$, 
and $r_{0}$ denotes the reflection with respect to the highest root of $E_{6}$.
Then we define $\sigma_{6}:Q^{\ast} \rightarrow Q^{\ast}$ by: 
\begin{equation*}
(\gamma_{1},\,\gamma_{2},\gamma_{3},\gamma_{4}) \mapsto 
(\psi_{3}(\gamma_{1}),\,\gamma_{4},\gamma_{2},\gamma_{3}). 
\end{equation*}
We know from \cite[\S3.2]{M} that $\Ni(E_{6}^{4}) \subset Q^{\ast}$ is 
stable under the action of $\sigma_{6}$, and hence 
$\sigma_{6} \in \Aut \Ni(E_{6}^{4})$. Since $\psi_{3}$ is fixed-point-free on $E_{6}$, 
we see that 
%
%%%%%%%%%%%%%
%%% eq:s6 %%%
%%%%%%%%%%%%%
%
\begin{equation} \label{eq:s6}
\rank \Ni(E_{6}^{4})^{\sigma_{6}} = 6.
\end{equation}
%
%=========================%
%     START SECTION 03    %
%=========================%
%
\section{Niemeier lattices and their automorphisms of order $3$.}
\label{sec:Lsigma6}
%
%==============================%
%     START SUBSECTION 0301    %
%==============================%
%
\subsection{Main result of \S\ref{sec:Lsigma6}.}
\label{sec:mainS3}
Let $L$ be a Niemeier lattice, and let $Q$ be its root lattice.
Let us consider the following conditions on $\tau \in \Aut L$: 
\begin{equation} \label{eq:tau}
\begin{cases}
|\tau|=3; \\[1.5mm]
\rank L^\tau \in 6\BZ; & \\[1.5mm]
\tau \notin G_{0}(L) = G_{0}(Q); 
\end{cases}
\end{equation}
recall that $G_{0}(L) = G_{0}(Q) \sbg \Aut L$ denotes the Weyl group of $Q$ 
(see \S\ref{subsec:Niemeier}). Notice that $\rank L^\tau = 0,\,6,\,12$, or $18$. 
For each $r \in \{0,\,6,\,12,\,18\}$, 
let us denote by $\CC_{r}$ the set of 
conjugacy classes in $\Aut L$ which contain 
elements $\tau \in \Aut L$ satisfying \eqref{eq:tau}
with $\rank L^\tau=r$. 
%
%%%%%%%%%%%%%%%%%
%%% rem:sigma %%%
%%%%%%%%%%%%%%%%%
%
\begin{rem} \label{rem:sigma}
Observe that all of $\sigma_{1},\,\dots,\,\sigma_{7}$ 
as in \S\ref{sec:review-Ni} satisfy \eqref{eq:tau}, and 
%
%%%%%%%%%%%%%%%%%%
%% table:sigma %%%
%%%%%%%%%%%%%%%%%%
%
\begin{equation} \label{table:sigma}
\begin{array}{c||c|c|c|c|c|c|c}
Q & A_2^{12} & D_4^{6} 
& D_4^{6} & D_4^{6} 
& A_{5}^{4}D_{4} & E_{6}^{4} & 
\{0\} \\[1mm] \hline 
\tau & \sigma_1 & \sigma_{2} 
& \sigma_{3} & \sigma_{4} 
& \sigma_{5} & \sigma_{6}
& \sigma_{7} \\[1mm] \hline\hline
\rank L^{\tau} 
& 6 & 0 & 6 & 6 & 6 & 6 & 0
\end{array}
\end{equation}
It can be easily checked that the automorphisms 
$\sigma_{3},\,\sigma_{4} \in \Aut \Ni(D_{4}^{6})$ are 
not conjugate to each other in $\Aut \Ni(D_{4}^{6})$; 
indeed, $\sigma_{3}$ fixes 18 elements in $\Delta = 
\bigl\{\alpha \in \Ni(D_{4}^{6}) \mid \pair{\alpha}{\alpha}=2 \bigr\}$, 
but $\sigma_{4}$ fixes no element in $\Delta$ 
(see also Propositions~\ref{prop:autoD4a} and \ref{prop:autoD4b} below). 
\end{rem}

The following is the main theorem in this section.
%
%%%%%%%%%%%%%%%%%%
%%% thm:mainS3 %%%
%%%%%%%%%%%%%%%%%%
%
\begin{thm}\label{thm:mainS3} 
Let $L$ be a Niemeier lattice, and let $Q$ be its root lattice.
\begin{enu}
\item If there exists $\tau \in \Aut L$ 
satisfying \eqref{eq:tau}, 
then $Q$ is isomorphic to one of the following: 
\begin{equation} \label{eq:Q}
\bigl\{ 0 \bigr\} ,\ A_1^{24},\ A_2^{12},\ 
A_3^{8},\ D_4^6,\ A_5^4D_4,\ A_6^4,\ D_6^4,\ E_6^4.
\end{equation}
Moreover, if $Q \neq \{0\}$, and 
there exists $\tau \in \Aut L$ which satisfies \eqref{eq:tau}, and 
preserves each indecomposable component of $Q$, 
then $Q = D_4^6$.

\item For each of $Q$'s in \eqref{eq:Q} and 
$r=0$, $6$, $12$, $18$, the cardinality $\# \CC_{r}$ of 
the set $\CC_{r}$ is given 
by the following table:
%
%%%%%%%%%%%%%%%%%%%%
%%% table:order3 %%%
%%%%%%%%%%%%%%%%%%%%
%
\begin{equation} \label{table:order3}
\begin{array}{|c||c|c|c|c|c|c|c|c|c|c|} \hline 
 Q & \{0\} &
 A_1^{24} & A_2^{12} & A_3^8 &
 D_4^6 & A_5^4 D_4 & A_6^4 & D_6^4 & E_6^4  \\ \hline\hline
 \#\CC_{0}
  & 1 & 0 & 0 & 0 & 1 & 0 & 0 & 0 & 0  \\ \hline
 \#\CC_{6} 
  & 1 & 0 & 1 & 0 & 2 & 1 & 0 & 0 & 1  \\ \hline
 \#\CC_{12}
  & 1 & 1 & 1 & 1 & 2 & 1 & 2 & 1 & 1  \\ \hline
 \#\CC_{18}
  & 0 & 0 & 0 & 0 & 0 & 0 & 0 & 0 & 0  \\ \hline
\end{array}
\end{equation}
In particular, 
\begin{enumerate}[\rm (i)]
\item there exists no $\tau \in \Aut L$ 
satisfying \eqref{eq:tau} with $\rank L^{\tau}=18$; 

\item if $Q \ne \{0\}$, and $\rank L^{\tau} \in \bigl\{0,\,6\bigr\}$, 
then $\tau$ is conjugate to one of 
$\sigma_{1},\,\dots,\,\sigma_{6}$ (see Remark~\ref{rem:sigma}). 
\end{enumerate}
\end{enu}
\end{thm}

We will prove Theorem \ref{thm:mainS3}\,(1) in \S\ref{subsec:main3}. 
In \S\ref{subsec:main32} and \S\ref{subsec:main32-D4}, 
for each Niemeier lattice $L=\Ni(Q)$ given in Theorem~\ref{thm:mainS3}\,(1), 
we will classify the automorphisms $\tau \in \Aut L$ 
satisfying \eqref{eq:tau}, up to conjugation, thereby proving 
Theorem \ref{thm:mainS3}\,(2); 
\begin{itemize}
\item 
the classification for $Q=\{0\}$ 
will be given in Proposition~\ref{prop:Leech}; 

\item 
the classification for $Q=A_5^4 D_4$ 
will be given in Proposition~\ref{prop:autoA5}; 

\item
the classification for 
$Q = A_1^{24}$, $A_3^8$, $A_6^4$, $D_6^4$, $A_2^{12}$, and $E_6^4$
will be given in Proposition~\ref{prop:autoA1}; 

\item 
the classification for $Q = D_4^6$ will be given in 
Propositions~\ref{prop:autoD4a} and \ref{prop:autoD4b}. 
\end{itemize}

%=============================%
%     START SUBSECTION 0302   %
%=============================%
%
\subsection{Automorphisms of root lattices of order $3$.}
\label{subsec:autQm}
It is obvious that $\Aut A_{1} = W(A_{1}) \cong \BZ_{2}$ 
does not have an element of order $3$.
%
%%%%%%%%%%%%%%%%%%%%
%%%%%lem:fixedr%%%%%
%%%%%%%%%%%%%%%%%%%%
%
\begin{lem} \label{lem:fixedr}
Let $R$ be a root lattice of type $A_{n}$ with $n \ge 2$, 
$D_{n}$ with $n \ge 5$, or $E_{6}$. 
\begin{enu}

\item
If $\ve \in \Aut R$ is of order $3$, then $\ve$ is contained 
in the Weyl group $W(R)$ of $R$, and 
\begin{equation*}
\rank R^{\ve}=
 \begin{cases}
 \text{\rm $n-2c$ for some $1 \le c \le (n+1)/3$, 
       if $R=A_{n}$}, \\[1.5mm]
 \text{\rm $n - 2c$ for some $1 \le c \le n/3$, 
       if $R=D_{n}$}, \\[1.5mm]
 \text{\rm $0$, $2$, or $4$, if $R=E_{6}$}.
 \end{cases}
\end{equation*}

\item If $\ve_1,\,\ve_2 \in \Aut R$ are of order $3$ 
(and hence $\ve_{1},\,\ve_{2} \in W(R)$ by (1)), 
and $\rank R^{\ve_1} = \rank R^{\ve_2}$, 
then $\ve_1$ and $\ve_{2}$ are conjugate to each other in $W(R)$. 
\end{enu}
\end{lem}

\begin{proof}
Recall that $\Aut R = W(R) : G_{1}(R)$. 
Since $G_{1}(R)$ ($=$ the group of Dynkin diagram automorphisms) 
does not have an element of order $3$ in these cases, 
we get $\ve \in W(R)$ if $\ve \in \Aut R$ is of order $3$. 

Assume that $R=A_{n}$ with $n \ge 2$. 
Since $W(A_{n}) \cong \FS_{n+1}$, 
we see that $\ve$ is equal to 
a product of mutually commutative $3$-cycles; 
notice that the number $c$ of $3$-cycles in $\ve$ satisfies 
$1 \le c \le (n+1)/3$. Then we have $\rank A_{n}^\ve = n-2c$, 
which proves part (1). Also, part (2) is obvious by this argument. 

Assume that $R=D_{n}$ with $n \ge 5$; 
recall that $W(D_{n}) \cong \BZ_2^{n-1} : \FS_n$. 
Write $\ve$ as: $\ve = z \rho$, 
where $z \in \BZ_{2}^{n-1}$ and $\rho \in \FS_{n}$. 
Then, $\ve$ is conjugate to $\rho$. Indeed, 
we see that $\rho$ is of order $3$, and 
$(z\rho^{-1})^{3}=1$. Hence, 
$(\rho z \rho^{-1})^{-1} \ve (\rho z \rho^{-1}) = \rho (z\rho^{-1})^3=\rho$. 
Thus we may assume from the beginning that $\ve \in \FS_{n}$. 
Then, $\ve$ is equal to a product of mutually commutative $3$-cycles; 
notice that the number $c$ of $3$-cycles in $\ve$ satisfies 
$1 \le c \le n/3$. It can be easily seen that $\rank D_{n}^\ve = n-2c$, 
which proves part (1). Part (2) is obvious by this argument.

The assertions for the case of $R=E_{6}$ follow from 
the character table of $\Aut E_6$, 
which can be obtained by the computer program ``MAGMA''. 
%
% Note: 
% \Aut E_6 \cong \BZ_2 \times U_4(2).\BZ_2$ 
% $U_{4}(2)$ is the unitary group of degree $4$ over $\BF_{2}$
%
%\footnote{For example, in order to get the character table 
%of $\Aut E_6$ by MAGMA, we can use the following command: \\[1.5mm]
%{\tt CharacterTable(AutomorphismGroup(Lattice("E",6)));}}
%
\end{proof}

We turn to the case of $D_{4}$; 
recall the definitions of $\omega,\,\vp,\,\psi \in \Aut D_{4}$ 
from \S\ref{subsec:root}. 
Define $P:=\langle W(D_{4}),\,\omega\rangle \cong 
W(D_{4}) : \langle \omega \rangle \lhd 
\Aut D_{4} \cong W(D_{4}) : \FS_{3}$; recall that 
$\vp \in W(D_{4})\omega \subset P$. 
We claim that an element $z \in W(D_{4})\omega$ 
is not conjugate to any element of $W(D_{4})\omega^{-1}$ in $P$. 
Indeed, since $P=\langle W(D_{4}),\,\omega\rangle$ and $W(D_{4}) \lhd P$, 
we see that $x^{-1} z x$ is contained in $W(D_{4})\omega$ for all $x \in P$. 
The assertion follows immediately from this fact and 
$W(D_{4}) \omega \cap W(D_{4})\omega^{-1} = \emptyset$. 
%
%%%%%%%%%%%%%%
%%% lem:D4 %%%
%%%%%%%%%%%%%%
%
\begin{lem} \label{lem:D4}
\mbox{}
\begin{enu}
\item
There are exactly three conjugacy classes 
of order $3$ elements in $\Aut D_{4}$, which are 
$\Con{\psi}{\Aut D_4}$, 
$\Con{\omega}{\Aut D_4}$, and 
$\Con{\vp}{\Aut D_4}$. 
Moreover, we have
\begin{equation}
\begin{cases}
\Con{\psi}{\Aut D_4} = \Con{\psi}{P}; \\[1.5mm]
\Con{\omega}{\Aut D_4} = 
\Con{\omega}{P} \sqcup \Con{\omega^{-1}}{P}; \\[1.5mm]
\Con{\vp}{\Aut D_4} = 
\Con{\vp}{P} \sqcup \Con{\vp^{-1}}{P}. 
\end{cases}
\end{equation}

\item Let $\ve \in \bigl\{\omega,\,\vp\bigr\}$ 
and $c \in \bigl\{1,\,-1\bigr\}$. Then, 
\begin{equation}
\Con{\ve^{c}}{P}=\bigl\{y^{-1} \ve^{c} y \mid y \in W(D_{4})\bigr\}.
\end{equation}
\end{enu}
\end{lem}

\begin{proof}
(1) First, we used the computer program ``MAGMA'' to obtain 
the character table of $\Aut D_{4}$, from which we see that 
$\Aut D_{4}$ has exactly three conjugacy classes of 
order $3$ elements in $\Aut D_{4}$; 
by \eqref{eq:D4-fixed}, we have 
$\Con{\omega}{\Aut D_4} \ne \Con{\vp}{\Aut D_4}$. 
Since $\Con{\psi}{\Aut D_4} \subset W(D_{4})$ 
(recall that $\psi \in W(D_{4}) \lhd \Aut D_{4}$), 
and since $\omega,\,\vp \not\in W(D_{4})$, 
it follows that $\Con{\psi}{\Aut D_4} \ne 
\Con{\omega}{\Aut D_4}$, $\Con{\vp}{\Aut D_4}$. 
Thus we conclude that 
$\Con{\psi}{\Aut D_4}$, 
$\Con{\omega}{\Aut D_4}$, and 
$\Con{\vp}{\Aut D_4}$ are distinct to each other. 
Also, by \eqref{eq:D4-fixed} and the fact that 
$\Con{\psi}{\Aut D_4} \subset W(D_{4})$, we see that 
$\vp^{-1} \in \Con{\vp}{\Aut D_4}$ and $\omega^{-1} \in \Con{\omega}{\Aut D_4}$. 

Let $\rho$ be the Dynkin diagram automorphism of $D_{4}$ which interchanges $\alpha_{3}$ and 
$\alpha_{4}$ with notation in \S\ref{subsec:root}. Then, $\Aut D_{4}=P \sqcup \rho P$, and hence 
\begin{equation*}
\Con{\ve}{\Aut D_4} = \Con{\ve}{P} \cup \Con{\rho^{-1}\ve\rho}{P}, 
\end{equation*}
where $\ve \in \bigl\{\psi,\,\omega,\,\vp\bigr\}$. 
If $\ve = \psi = r_{1}r_{2}$, then we see that $\rho^{-1} \psi \rho = \psi$, 
and hence $\Con{\psi}{P} = \Con{\rho^{-1}\psi\rho}{P}$. Thus we get
$\Con{\psi}{\Aut D_4} = \Con{\psi}{P}$. 
If $\ve = \omega$ or $\vp$, then we see that 
$\rho^{-1} \ve \rho \in W(D_{4})\omega^{-1}$ 
since $\rho^{-1} \omega \rho = \omega^{-1}$. 
Hence, by the argument preceding this lemma, 
$\ve$ is not conjugate to $\rho^{-1} \ve \rho$ in $P$, 
which implies that $\Con{\ve}{P} \ne \Con{\rho^{-1}\ve\rho}{P}$; 
\begin{equation*}
\Con{\ve}{\Aut D_4} = \Con{\ve}{P} \sqcup \Con{\rho^{-1}\ve\rho}{P}.
\end{equation*}
Since $\ve^{-1} \in \Con{\ve}{\Aut D_4}$ as seen above, and 
$\ve^{-1} \notin \Con{\ve}{P}$, it follows that 
$\ve^{-1} \in \Con{\rho^{-1}\ve\rho}{P}$, and hence 
$\Con{\rho^{-1}\ve\rho}{P} = \Con{\ve^{-1}}{P}$. Thus we have proved part (1). 

(2) If $\ve=\omega$, then the assertion is obvious 
since $P=\langle W(D_{4}),\,\omega\rangle$ and $W(D_{4}) \lhd P$. 
Assume that $\ve=\vp$. The cyclic group $\langle \omega \rangle$ (of order $3$) 
acts on $\Con{\vp^{c}}{P}$ by conjugation. 
Since $\# \Con{\vp}{\Aut D_4} = 16$ by the character table obtained by MAGMA, 
we see that $\# \Con{\vp^{c}}{P} = \# \Con{\vp}{\Aut D_4}/2 = 8$.
Hence there exists 
$\vp' \in \Con{\vp^{c}}{P}$ such that $\omega^{-1}\vp'\omega=\vp'$. 
Then we see that $\Con{\vp^{c}}{P}=\Con{\vp'}{P}=
\bigl\{y^{-1} \vp' y \mid y \in W(D_{4})\bigr\}$. 
Therefore we obtain 
$\Con{\vp^{c}}{P}=\bigl\{y^{-1} \vp^{c} y \mid y \in W(D_{4})\bigr\}$, 
as desired.
\end{proof}

%=============================%
%     START SUBSECTION 0303   %
%=============================%
%
\subsection{Some technical lemmas.}
\label{subsec:lemmas}

In this subsection, we show some basic lemmas, 
which will be needed in the proof of Theorem~\ref{thm:mainS3}.
%
%%%%%%%%%%%%%%
%%% lem:Z2 %%%
%%%%%%%%%%%%%%
%
\begin{lem} \label{lem:Z2}
Let $G$ be a finite group. 
If $G_{1} \lhd G$ and $|G_{1}|=2$, then 
$G_{1}$ is contained in the center of $G$. 
Moreover, the canonical projection 
$G \twoheadrightarrow G/G_{1}$ induces 
a bijection from the set of conjugacy classes 
of order $3$ elements in $G$ onto the ones in $G/G_{1}$.
\end{lem}

\begin{proof}
An easy exercise.
\end{proof}
%
%%%%%%%%%%%%%%%%
%%% lem:rank %%%
%%%%%%%%%%%%%%%%
%
\begin{lem}\label{lem:rank} 
Let $L$ be a lattice, let $\tau \in \Aut L$ be of order $3$, and 
let $Q$ be a sublattice of $L$ such that $\tau(Q)=Q$. 
\begin{enu}
\item If $\rank L=\rank Q$, 
then $\rank L^\tau=\rank Q^\tau$.

\item If there exist sublattices 
$R_{1}$, $R_{2}$, $R_{3}$, $R_{4}$ of $Q$ such that 
$Q=\bigoplus_{q=1}^{4} R_q$, and 
$\tau(R_{1})=R_{2}$, 
$\tau(R_{2})=R_{3}$, 
$\tau(R_{3})=R_{1}$, 
$\tau(R_{4})=R_{4}$, then 
$\rank Q^\tau=\rank R_1 + \rank R_{4}^{\tau}$. 
\end{enu}
\end{lem}

\begin{proof}
Part (1) can be easily verified as follows: 
$\rank L^{\tau} = \dim (L \otimes_{\BZ} \BR)^{\tau} = 
\dim (Q \otimes_{\BZ} \BR)^{\tau}=\rank Q^{\tau}$. 
Let us show part (2). Note that 
$(R_{1} \oplus R_{2} \oplus R_{3})^{\tau}=
 \bigl\{ x + \tau(x) + \tau^{2}(x) \mid x \in R_{1} \bigr\}$.
Thus we have an isomorphism of free $\BZ$-modules from 
$R_{1}$ onto $(R_{1} \oplus R_{2} \oplus R_{3})^{\tau}$ 
defined by: $x \mapsto x+\tau(x)+\tau^{2}(x)$; 
notice that this map does not preserve the $\BZ$-bilinear forms. 
Hence we obtain an isomorphism of free $\BZ$-modules 
from $Q^{\tau}$ onto $R_{1} \oplus R_{4}^{\tau}$, 
which implies that $\rank Q^\tau=\rank R_1 + \rank R_{4}^{\tau}$, 
as desired. 
\end{proof}
%
%%%%%%%%%%%%%%%%
%%% lem:conj %%%
%%%%%%%%%%%%%%%%
%
\begin{lem} \label{lem:conj}
Let $L$ be a Niemeier lattice, and let $Q$ be its root lattice.
Let $\tau \in \Aut L$ be of order~$3$.
Let $R_m$, $1\le m\le 4$, be root sublattices of $Q$ 
(not necessarily, indecomposable) such that
\begin{equation*}
Q=\bigoplus_{m=1}^{4} R_m, \qquad 
\tau (R_1)=R_2, \quad \tau (R_2)=R_3, \quad 
\tau (R_3)=R_1,\quad \tau (R_4)=R_4. 
\end{equation*}
Let $w \in W(R_1 \oplus R_2 \oplus R_3)=\prod_{m=1}^{3} W(R_m) \sbg \Aut L$.
If $w\tau$ is of order $3$, then $w\tau$ is conjugate to $\tau$. 
\end{lem}

\begin{proof}
First, we make some remarks. 
It is obvious that for every $1 \le i,\,j \le 3$ with $i \ne j$, 
$w_{i}w_{j} = w_{j}w_{i}$ for all $w_{i} \in W(R_{i})$ and $w_{j} \in W(R_{j})$. 
Also, we see that $R_{1} \cong R_{2} \cong R_{3}$, and hence 
$W(R_{1}) \cong W(R_{2}) \cong W(R_{3})$. Remark that 
for $w_{i} \in W(R_{i})$, $i=1,\,2,\,3$, and $p \in \BZ_{\ge 0}$, 
the element $w_{i}^{\tau^{p}} := \tau^{-p} w_{i} \tau^{p}$ is contained in 
$W(R_{\ol{i-p}})$, 
where $\ol{\phantom{v}}:\BZ \twoheadrightarrow \BZ/3\BZ = 
\bigl\{1,\,2,\,3\bigr\}$ is the canonical projection.

Now, let $w \in W(R_1 \oplus R_2 \oplus R_3)=
\prod_{m=1}^{3} W(R_m)$ be such that $w\tau$ is of order $3$, 
and write it as: $w=w_1w_2w_3$ with $w_i \in W(R_i)$, $1 \le i \le 3$. 
Since $1=(w\tau)^3$, and $\tau^{3}=1$, we have 
\begin{align*}
1 & = (w\tau)(w\tau)(w\tau) 
    = (w_1w_2w_3\tau)(w_1w_2w_3\tau)(w_1w_2w_3\tau) \\
  & = (w_1w_2w_3\tau)(w_1w_2w_3\tau)\tau(w_1^{\tau}w_2^{\tau}w_3^{\tau}) \\
  & = (w_1w_2w_3\tau)\tau^{2}(w_1^{\tau^2}w_2^{\tau^2}w_3^{\tau^2})
      (w_1^{\tau}w_2^{\tau}w_3^{\tau}) \\
  & = (w_1w_2w_3)(w_1^{\tau^2}w_2^{\tau^2}w_3^{\tau^2})
      (w_1^{\tau}w_2^{\tau}w_3^{\tau}) \\
  & = \underbrace{(w_1w_3^{\tau^2}w_2^{\tau})}_{\in W(R_{1})} 
      \underbrace{(w_2w_1^{\tau^2}w_3^{\tau})}_{\in W(R_{2})}
      \underbrace{(w_3w_2^{\tau^2}w_1^{\tau})}_{\in W(R_{3})}.
\end{align*}
Thus, 
\begin{equation*}
\underbrace{w_{1}w_{3}^{\tau^2} w_{2}^\tau}_{\in W(R_{1})}=
\underbrace{w_{2}w_{1}^{\tau^{2}}w_{3}^{\tau}}_{\in W(R_{2})}=
\underbrace{w_{3}w_{2}^{\tau^{2}}w_{1}^{\tau}}_{\in W(R_{3})}=1.
\end{equation*}
Here we set $u:=w_1w_{2} w_1^{\tau^2} \in W(R_1 \oplus R_2 \oplus R_3)=
\prod_{m=1}^{3} W(R_m)$. Then we have
\begin{align*}
u \tau u^{-1} 
 & = (w_1w_{2} w_1^{\tau^2}) \tau (w_1w_{2} w_1^{\tau^2})^{-1} 
   = (w_1w_{2} w_1^{\tau^2}) \tau 
     \{(w_1^{\tau^2})^{-1} w_{2}^{-1}w_{1}^{-1}\} \\
 & = (w_1w_{2} w_1^{\tau^2}) \tau 
     (w_{3}^{\tau}w_{2})w_{2}^{-1}w_{1}^{-1} 
   = (w_1w_{2} w_1^{\tau^2}) \tau 
      w_{3}^{\tau}w_{1}^{-1} \\
 & = (w_1w_{2} w_1^{\tau^2}) w_{3} \tau w_{1}^{-1} \\
 & = (w_1w_{2}w_{3}) w_1^{\tau^2} \tau w_{1}^{-1} 
   \quad \text{since $w_1^{\tau^2} \in W(R_{2})$ and $w_{3} \in W(R_{3})$} \\
 & = w\tau. 
\end{align*}
Thus we have proved the lemma. 
\end{proof}
%
%=============================%
%     START SUBSECTION 0304   %
%=============================%
%
\subsection{Proof of Theorem \ref{thm:mainS3}\,(1).}
\label{subsec:main3}

Let $L=\Ni(Q)$ be a Niemeier lattice such that 
$Q \ne \{0\}$ (i.e., $L \ne \Lambda$).
Let $Q=\bigoplus_{m=1}^{n} Q_{m}$ be 
the decomposition of $Q$ into 
its indecomposable components.
Assume that $\Aut L$ has an element $\tau$ 
satisfying \eqref{eq:tau}. 
%
%%%%%%%%%%%%%%
%%% c:m2-1 %%%
%%%%%%%%%%%%%%
%
\begin{claim} \label{c:m2-1}
If $\tau \in G_{0}(L):G_{1}(L)$, that is, if
$\tau(Q_{m})=Q_{m}$ for all $1 \le m \le n$ 
(see (\ref{eq:KQ-L})), then $Q=D_{4}^{6}$.
\end{claim}

\noindent
{\it Proof of Claim~\ref{c:m2-1}.}
Write $\tau$ uniquely as: $\tau=\tau_{0}\tau_{1}$ 
with $\tau_{0} \in G_{0}(L)$ and $\tau_{1} \in G_{1}(L)$; 
notice that $|\tau_{1}|=1$ or $3$ since the Weyl group 
$G_{0}(L)$ is a normal subgroup of $\Aut L$. 
Because $\tau \notin G_{0}(L)$ by \eqref{eq:tau}, 
it follows immediately that 
$\tau_{1} \ne 1$, and hence $\tau_{1}$ is of order $3$.
Since $\tau_{1} \in G_{1}(L) \sbg
G_{1}(Q)=\prod_{m=1}^{n} G_{1}(Q_{m})$ (see \eqref{eq:G01Q}), 
there exists $1 \le m \le n$ such that $G_{1}(Q_{m})$ contains 
a Dynkin diagram automorphism of order $3$, 
which implies that $Q_{m}$ is of type $D_{4}$. 
Therefore we see from the list of Niemeier lattices 
(see \cite[Chapter~16, Table~16.1]{CS} for example) that 
$Q=A_5^4D_4$ or $Q=D_4^6$. Since $|G_1 (L)|=2$ 
if $Q=A_5^4D_4$ (see table \eqref{table:G1G2}), and since 
$\tau_{1} \in G_{1}(L)$ is of order $3$, 
we obtain $Q=D_4^6$, as desired. \bqed

\vsp

We next assume that $\tau \notin G_{0}(L):G_{1}(L)$, 
or equivalently, $\tau(Q_{m}) \ne Q_{m}$ for some $1 \le m \le n$.
Under the notation in Remarks~\ref{rem:G2} and \ref{rem:G2L}, 
$\Phi(\tau) \in \Sym \SC_{Q}$ 
acts on the set $\SC_{Q}=\bigl\{Q_{m} \mid 1 \le m \le n\bigr\}$ 
of the indecomposable components of $Q$ nontrivially, 
which implies that $\Phi(\tau)$ is of order $3$. 
Therefore there exist at least $3$ mutually isomorphic 
components in $\SC_Q$; by the list of Niemeier lattices 
\cite[Chapter~16, Table 16.1]{CS}, $Q$ is one of the following: 
\begin{equation*}
A_1^{24},\ A_2^{12},\ A_3^8,\ A_4^6,\ 
A_5^4D_4,\ D_4^6,\ A_6^4,\ A_8^3,\ 
D_6^4,\ E_6^4,\ D_8^3,\ E_8^3.
\end{equation*}
%
%%%%%%%%%%%%%%
%%% c:m2-2 %%%
%%%%%%%%%%%%%%
%
\begin{claim} \label{c:m2-2}
$Q$ cannot be $A_4 ^6$, $A_8 ^3$, $D_8 ^3$, nor $E_8 ^3$.
\end{claim}

\noindent
{\it Proof of Claim~\ref{c:m2-2}.}
Suppose first that 
$Q=X^{3}$ with $X = A_{8}$, $D_{8}$, or $E_{8}$. 
Because $\Phi(\tau) \in \Sym \SC_{Q}$ is of order $3$, 
we see that $\tau$ cyclically permutes 
the $3$ indecomposable components 
$Q_{1}$, $Q_{2}$, $Q_{3}$ of $Q$ as:
$\tau(Q_{1})=Q_{2}$, $\tau(Q_{2})=Q_{3}$, 
$\tau(Q_{3})=Q_{1}$. Therefore it follows immediately 
from Lemma \ref{lem:rank} that 
$\rank L^\tau = \rank Q^{\tau} = \rank X = 8$, 
which contradicts the fact that $\rank L^{\tau}$ is divisible by $6$. 

Suppose next that $Q=A_4^6$. 
By \cite[p.\,408]{CS}, 
$\Phi(\Aut L) \cong G_{2} (L) \cong PGL_2(5) \ (\cong \FS_{5})$, 
where $PGL_{2}(5)$ is the projective general linear group of degree $2$ 
over $\BF_{5}$, and it acts 
on the set $\SC_{Q} \cong \bigl\{\infty,\,0,\,1,\,2,\,3,\,4\bigr\}$ as 
linear fractional transformations (see also \cite[Chapter~10, \S1]{CS}); 
note that $PGL_2(5) \ (\cong \FS_{5})$ has a unique conjugacy class 
of order $3$ elements, and 
$\left(\begin{smallmatrix} 0 & -1 \\ 1 & -1 \end{smallmatrix}\right) \, 
\mathrm{mod} \ \BF_{5}^{\times}$ 
is a representative for it, which acts on $\SC_{Q}$ as $(\infty01)(243)$. 
Since $\Phi(\tau) \in \Sym \SC_{Q}$ is of order $3$, 
it follows immediately that $\Phi(\tau)$ is conjugate to 
$\left(\begin{smallmatrix} 0 & -1 \\ 1 & -1 \end{smallmatrix}\right)\, 
\mathrm{mod} \ \BF_{5}^{\times}$, and hence acts on $\SC_{Q}$ 
as a product of two mutually commutative $3$-cycles. 
Therefore, $Q$ has $3$ components 
$R_{1}$, $R_{2}$, $R_{3}$ of type $A_{4}^{2}$, 
which are cyclically permuted by $\tau$ as:
$\tau(R_{1})=R_{2}$, $\tau(R_{2})=R_{3}$, 
$\tau(R_{3})=R_{1}$. 
Therefore it follows immediately 
from Lemma \ref{lem:rank} that 
$\rank L^\tau=\rank A_{4}^{2}=8$, 
which contradicts \eqref{eq:tau}. \bqed

\vsp

Thus we have proved Theorem \ref{thm:mainS3}\,(1). 

%=============================%
%     START SUBSECTION 0305   %
%=============================%
%
\subsection{Proof of Theorem \ref{thm:mainS3}\,(2): 
Case that $Q \ne D_{4}^{6}$.}
\label{subsec:main32}

If $L$ is the Leech lattice $\Lambda$, then 
$\Aut L = \Aut \Lambda$ is of type $2.\Co_1$
(see \S\ref{subsec:Niemeier}). 
We know from \cite[$\Co_{1}$]{ATLAS} that 
$\Co_{1}$ has exactly 4 conjugacy classes 
$3A$, $3B$, $3C$, and $3D$ of order $3$ elements 
with notation therein. 
Hence, by Lemma~\ref{lem:Z2}, 
% \cite[p.\,184, Ind]{ATLAS},
$\Aut \Lambda$ has exactly 
4 conjugacy classes of order $3$ elements, 
which we denote also by $3A$, $3B$, $3C$, and $3D$.
We deduce from the character table in 
\cite[$\Co_{1}$]{ATLAS} that 
if $\tau \in 3A$ (resp., $3B$, $3C$, $3D$) 
in $\Aut \Lambda \cong 2.\Co_{1}$, then 
$\rank \Lambda^\tau = 0$ (resp., $12$, $6$, $8$). 
Thus we obtain the following proposition. 
%
%%%%%%%%%%%%%%%%%%
%%% prop:Leech %%%
%%%%%%%%%%%%%%%%%%
%
\begin{prop} \label{prop:Leech}
An automorphism $\tau \in \Aut \Lambda$ satisfies (\ref{eq:tau}) 
if and only if it is contained 
in the conjugacy classes 
$3A$, $3B$, or $3C$ in $\Aut \Lambda \cong 2.\Co_{1}$. 
Moreover, if $\tau \in 3A$ (resp., $3B$, $3C$), 
then $\rank \Lambda^\tau = 0$ (resp., $12$, $6$). 
\end{prop}

Next, let us consider the case that 
$Q$ is neither $\{0\}$ nor $D_{4}^{6}$. 
In these cases, we see from Theorem~\ref{thm:mainS3}\,(1) 
that if $\tau \in \Aut L$ satisfies \eqref{eq:tau}, then 
$\Phi(\tau) \in \Phi(\Aut L) \ (\cong G_{2}(L)) \sbg \Sym \SC_{Q}$ 
acts on the set $\SC_{Q}=\bigl\{Q_{m} \mid 1 \le m \le n\bigr\}$ 
of indecomposable components of $Q=\bigoplus_{m=1}^{n} Q_{m}$ 
nontrivially, and hence $\Phi(\tau)$ is of order $3$ 
(see Remarks~\ref{rem:G2} and \ref{rem:G2L}, along with \eqref{eq:AutQ}).
%
%%%%%%%%%%%%%%%%%%
%%% lem:autoA5 %%%
%%%%%%%%%%%%%%%%%%
%
\begin{prop} \label{prop:autoA5} 
Let $L=\Ni(Q)$ be the Niemeier lattice with $Q=A_{5}^{4}D_{4}$. 
Enumerate $\SC_{Q}=\bigl\{Q_{m} \mid 1 \le m \le 5\bigr\}$ as 
$Q_{1} \cong Q_{2} \cong Q_{3} \cong Q_{4} \cong A_{5}$ and 
$Q_{5} \cong D_{4}$.

\begin{enu}

\item
If $\tau \in \Aut L$ satisfies \eqref{eq:tau}, 
then $\Phi(\tau) \in \Sym \SC_{Q}$ fixes 
$Q_{5} \in \SC_{Q}$, and acts on 
$\bigl\{Q_{m} \mid 1 \le m \le 4\bigr\} \subset \SC_{Q}$ 
as a $3$-cycle. If $Q_{k}$ is a unique element in 
$\bigl\{Q_{m} \mid 1 \le m \le 4\bigr\}$ fixed by $\Phi(\tau)$, 
then either (a) or (b) below holds:
\begin{enumerate}[\rm (a)]
\item
$\rank Q_{k}^{\tau} = 1$, 
$\tau|_{Q_{5}} \in \Aut D_{4}$ 
is conjugate to $\vp$ in $\Aut D_{4}$, and $\rank L^{\tau}=6$; 
\item
$\tau|_{Q_{k}} = \id$, 
$\tau|_{Q_{5}} \in \Aut D_{4}$ is conjugate to $\omega$ 
in $\Aut D_{4}$, and $\rank L^{\tau}=12$.
\end{enumerate}

\item
We have $\#\CC_{0}=\#\CC_{18}=0$, and 
$\#\CC_{6}=\#\CC_{12}=1$.
\end{enu}
\end{prop}

\begin{proof}
(1) It is obvious that $\Phi(\tau) \in \Sym \SC_{Q}$ fixes $Q_{5} \in \SC_{Q}$. 
Since $G_2 (L)$ is isomorphic to $\FS_4$ by table~\eqref{table:G1G2}, 
it follows that $\Phi(\tau) \in \Sym \SC_{Q}$ acts on 
$\bigl\{Q_{m} \mid 1 \le m \le 4\bigr\} \subset \SC_{Q}$ as a $3$-cycle. 
We see from Lemma~\ref{lem:rank}, 
along with Lemmas~\ref{lem:fixedr}, \ref{lem:D4}\,(1), and \eqref{eq:D4-fixed}, that
\begin{equation*}
6\BZ \ni 
\rank L^{\tau} = \rank Q^{\tau} = 
\underbrace{\rank Q_{k}^{\tau}}_{=1,\,3,\,\text{or}\,5} + 
\underbrace{\rank Q_{l}}_{=5} + 
\underbrace{\rank Q_{5}^{\tau}}_{=0,\,2,\,\text{or}\,4},
\end{equation*}
where $1 \le l \le 4$ with $l \ne k$. 

\begin{claim*}
$\tau|_{Q_{5}} \in \Aut D_{4}$ is not contained in $W(D_{4})$.
\end{claim*}

\noindent
{\it Proof of Claim.}
We use the glue code as in \S\ref{subsec:s5}; 
here, for simplicity of notation, we write the coset 
$[a_1a_2\cdots a_5]+Q$ as $[a_1a_2 \cdots a_5]$. 
By \S\ref{subsec:s5}, the glue code $L/Q$ consists of 
\begin{equation*}
s_1 [33001] + s_2 [30302] + s_3[30033] + 
t_1 [20240] + t_2 [22400] + t_3[24020]
\end{equation*}
with $0 \le s_1,\,s_2,\,s_3 \le 1$ and 
$0 \le t_1,\,t_2,\,t_3 \le 2$. We see that 
\begin{align*}
B & :=\bigl\{ [a_1a_2 \cdots a_5] \in L/Q \mid 
  \text{$a_{m} \in \bigl\{0,\,3\bigr\}$ for all $1 \le m \le 4$}\bigr\} \\
  & = 
  \bigl\{ s_1 [33001] + s_2 [30302] + s_3[30033] \mid 
    0 \le s_1,\,s_2,\,s_3 \le 1
  \bigr\} \\
  & =
  \bigl\{ [00000], \, [33001], \, [30302], \, [30033], \, 
  [03303], \, [03032], \, [00331], \, [33330]
  \bigr\}.
\end{align*}
Observe that the subset $B$ above is stable under the induced action of 
$\Aut L$ on $L/Q$, because the induced action of 
$\Aut A_{5}$ on $A_{5}^{\ast}/A_{5}$ 
fixes $\ol{[0]},\,\ol{[3]} \in A_{5}^{\ast}/A_{5}$ 
(see \S\ref{subsec:root} for the notation). 

Now, suppose, for a contradiction, that $\tau|_{Q_{5}} \in W(D_{4})$. 
Then the induced action of $\tau$ on $L/Q$ does not change 
the fifth entry of any $[a_1a_2 \cdots a_5] \in L/Q$. 
Since the subset $B$ above is stable under the induced action of $\Aut L$ on $L/Q$
as mentioned above, it follows that $\tau [33001]$ is equal to 
an element of $B$ whose fifth entry is equal to $1$. 
Namely, $\tau [33001] = [33001]$ or $[00331]$. Similarly, 
$\tau [00331] = [33001]$ or $[00331]$. 
Hence the subset $\bigl\{[33001],\,[00331]\bigr\}$ of $B$ 
is stable under the action of $\tau$. 
Since $|\tau|=3$, it follows immediately that 
$\tau$ fixes both $[33001]$ and $[00331]$. 
However, this contradicts the fact that $\Phi(\tau)$ acts on 
$\bigl\{Q_{m} \mid 1 \le m \le 4\bigr\} \subset \SC_{Q}$ as a $3$-cycle. 
Thus we have proved Claim. \bqed

\vsp 

By Claim and Lemma~\ref{lem:D4}\,(1), 
$\tau|_{Q_{5}}$ is conjugate to 
either $\vp$ or $\omega$ in $\Aut D_{4}$. 
If $\tau|_{Q_{5}}$ is conjugate to $\vp$ (resp., $\omega$), 
then $\rank Q_{5}^{\tau}=0$ (resp., $=2$), 
and hence $\rank Q_{k}^{\tau}=1$ (resp., $=5$) 
since $\rank L^{\tau} \in 6\BZ$. 
Thus we have proved part (1). 

(2) Part (1) implies that $\#\CC_{0}=\#\CC_{18}=0$. 
Let us show that $\#\CC_{6}=\#\CC_{12}=1$. 
We know from Remark~\ref{rem:sigma} that 
$\sigma_{5} \in \Aut L$ as in \S\ref{subsec:s5} 
satisfies \eqref{eq:tau} with $\rank L^{\sigma_{5}}=6$, 
which implies that $\#\CC_{6} \ge 1$. 
Recall from \S\ref{subsec:s5} that 
$\sigma_{5}$ fixes $Q_{1} \cong A_{5}$ and 
$Q_{5} \cong D_{4}$, and 
permutes $Q_{2} \cong Q_{3} \cong Q_{4} \cong A_{5}$ 
cyclically, and that 
$\sigma_{5}|_{Q_{1}} = \psi \in W(A_{5}) \cong W(Q_{1})$ 
and $\sigma_{5}|_{Q_{5}} = \vp \in W(D_{4}) \omega$. 
Let $w \in W(D_{4}) \cong W(Q_{5})$ 
be such that $w\vp = \omega$.
Then we see that $\sigma:=\psi^{-1} w \sigma_{5} \in \Aut L$ 
satisfies \eqref{eq:tau} with $\rank L^{\tau} = 12$. 
Thus we get $\#\CC_{12} \ge 1$.

%%%%%%%%%%%%%%%%%%%%%%%%%%
Next, we show $\#\CC_{6}=1$ and $\# \CC_{12}=1$; 
we give a proof only for $\#\CC_{6}=1$ since $\#\CC_{12}=1$
can be shown similarly. 
Assume that $\tau,\,\tau' \in \Aut L$ satisfy 
\eqref{eq:tau} with $\rank L^{\tau}=\rank L^{\tau'}=6$. 
Write $\tau$ and $\tau'$ as: 
$\tau=\tau_{0}\tau_{H}$ and $\tau'=\tau'_{0}\tau'_{H}$ 
with $\tau_{0},\,\tau'_{0} \in G_{0}(L)$ and 
$\tau_{H},\,\tau'_{H} \in H(L)$ (see \eqref{eq:HL});
remark that $|\tau_{H}|=|\tau_{H}'|=3$ and 
$|\pi_{2}(\tau_{H})|=|\pi_{2}(\tau'_{H})|=3$.
Because $G_{2}(L) \cong \FS_{4}$ has a unique conjugacy class 
consisting of order $3$ elements, $\pi_{2}(\tau_{H})$ is conjugate to 
$\pi_{2}(\tau'_{H})$ in $G_{2}(L)$. 
Since $G_{1}(L) \cong \BZ_{2}$ in this case 
(see table~\eqref{table:G1G2}), and 
since the sequence in \eqref{eq:exact} is exact, 
it follows from Lemma~\ref{lem:Z2} that 
$\tau_{H}$ is conjugate to $\tau_{H}'$ in $H(L)$. 
Let $h \in H(L)$ be such that $h^{-1}\tau_{H}h=\tau_{H}'$. 
Then, $\tau' = \tau_0' \tau_H' = 
\tau_0' (h^{-1}\tau_{H}h) = 
h^{-1} \{(h \tau_{0}' h^{-1}) \tau_{H}\} h$; 
note that $h \tau_{0}' h^{-1} \in G_{0}(L)$. 
Hence, by replacing $\tau'$ with 
$(h \tau_{0}' h^{-1}) \tau_{H}$, 
we may assume from the beginning 
that $\tau_{H}'=\tau_{H}$. 
In particular, we have $\Phi(\tau)=\Phi(\tau') \in \Sym \SC_{Q}$; 
we may assume without loss of generality that 
$\Phi(\tau)=\Phi(\tau')$ fixes $Q_{1} \cong A_{5}$ and $Q_{5} \cong D_{4}$, 
and permutes $Q_{2} \cong Q_{3} \cong Q_{4} \cong A_{5}$ cyclically as: 
$Q_{2} \to Q_{3} \to Q_{4} \to Q_{2}$. 

By part (1), 
$\tau|_{Q_{5}} \in \Aut D_{4}$ and 
$\tau'|_{Q_{5}} \in \Aut D_{4}$ are 
conjugate to $\vp$ in $\Aut D_4$. 
Also, since $\tau_{H}'=\tau_{H}$, 
we deduce from Lemma~\ref{lem:D4}\,(1) that 
both $\tau|_{Q_{5}}$ and $\tau'|_{Q_{5}}$ 
are contained in $\Con{\vp^{c}}{P}$, where $c=1$ or $-1$. 
Hence, by Lemma~\ref{lem:D4}\,(2), 
there exists $y \in W(D_{4}) \cong W(Q_{5})$ such that 
$y^{-1}(\tau|_{Q_{5}})y = \tau'|_{Q_{5}}$. 
Then we see by a direct computation 
that $y^{-1}\tau y = \tau_{0}'' \tau_{H}$ 
for some $\tau_{0}'' \in G_{0}(L)$ such that 
$\tau_{0}''|_{Q_{5}} = \tau_{0}'|_{Q_{5}}$. 
Hence we may assume from the beginning that 
$\tau_{0}|_{Q_{5}} = \tau_{0}'|_{Q_{5}}$. 
Write $\tau_{0} \in G_{0}(L)$ and 
$\tau_{0}' \in G_{0}(L)$ as:
\begin{equation*}
\tau_{0} = (x_{1},\,x_{2},\,x_{3},\,x_{4},\,x_{5}), \qquad
\tau_{0}' = (x_{1}',\,x_{2}',\,x_{3}',\,x_{4}',\,x_{5}).
\end{equation*}
with $x_{m},\,x_{m}' \in W(A_{5})$, $1 \le m \le 4$, and 
$x_{5} \in W(D_{5})$. Set 
\begin{equation*}
w:=(1,\,x_{2}'x_{2}^{-1},\,x_{3}'x_{3}^{-1},\,
    x_{4}'x_{4}^{-1},\,1) \in G_{0}(L).
\end{equation*}
We deduce from Lemma~\ref{lem:conj} that 
$w\tau=w\tau_{0}\tau_{H}$ is conjugate to 
$\tau=\tau_{0}\tau_{H}$. 
Thus, by replacing $\tau$ by $w\tau$, 
we may assume that $x_{m}=x_{m}'$ 
also for all $2 \le m \le 4$. 
Since $\rank L^{\tau}=\rank L^{\tau'}=6$, we have 
$\rank Q_{1}^{\tau}=\rank Q_{1}^{\tau'}=1$ by part (1). 
By Lemma~\ref{lem:fixedr}, 
there exists $z \in W(A_{5}) \cong W(Q_{4})$ 
such that $z^{-1} \bigl(\tau|_{Q_{1}}\bigr) z = 
\bigl(\tau'|_{Q_{1}}\bigr)$. 
Then we see that $z^{-1}\tau z = \tau'$. 
Thus we have proved part (2). 
This completes the proof of Proposition~\ref{prop:autoA5}.
\end{proof}
%
%%%%%%%%%%%%%%%%%%%
%%% prop:autoA1 %%%
%%%%%%%%%%%%%%%%%%%
%
\begin{prop} \label{prop:autoA1}
Let $L=\Ni(Q)$ be the Niemeier lattice with $Q=A_1^{24}$, 
$A_3^8$, $A_6^4$, $D_6^4$, $A_2^{12}$, or $E_{6}^{4}$. 
Set $n:=\# \SC_{Q}$. 
\begin{enu}
\item
If $\tau \in \Aut L$ satisfies (\ref{eq:tau}), then 
$\Phi(\tau) \in \Sym \SC_{Q}$ fixes exactly 
$k:=n/4$ elements in $\SC_{Q}=
\bigl\{Q_{m} \mid 1 \le m \le n\bigr\}$; 
we may assume that 
%
%%%%%%%%%%%%%
%%% eq:Qm %%%
%%%%%%%%%%%%%
%
\begin{equation} \label{eq:Qm}
\tau(Q_{m})=
 \begin{cases}
   Q_{m+k} & \text{\rm if $1 \le m \le 2k$}, \\[1.5mm]
   Q_{m-2k} & \text{\rm if $2k+1\le m \le 3k$}, \\[1.5mm]
   Q_m & \text{\rm if $3k+1 \le m \le n$}.
 \end{cases}
\end{equation}
If the automorphism group of an indecomposable component of $Q$ has 
a fixed-point-free element of order $3$, i.e., 
if $Q=A_{2}^{12}$ or $E_{6}^{4}$ 
(see Lemma~\ref{lem:fixedr}), then either (a) or (b) below holds: 
\begin{enumerate}[\rm (a)]

\item $\tau$ acts on all of $Q_{m}$'s, $3k+1 \le m \le n$, 
trivially, and $\rank L^{\tau} = 12$; 

\item $\tau$ acts on all of $Q_{m}$'s, $3k+1 \le m \le n$, 
fixed-point-freely, and $\rank L^{\tau} = 6$. 

\end{enumerate}
Otherwise, $\tau$ acts on all of $Q_{m}$'s, $3k+1 \le m \le n$, 
trivially, and $\rank L^{\tau} = 12$.

\item
We have $\#\CC_{0}=\#\CC_{18}=0$, and 
\begin{equation*}
\#\CC_{12}=
\begin{cases}
2 & \text{\rm if $Q=A_{6}^{4}$}, \\
1 & \text{\rm otherwise}, 
\end{cases}
\qquad
\#\CC_{6}=
\begin{cases}
1 & \text{\rm if $Q=A_{2}^{12}$ or $E_{6}^{4}$}, \\
0 & \text{\rm otherwise}.
\end{cases}
\end{equation*}
\end{enu}
\end{prop}

\begin{proof}
(1) First, let us check that 
$\Phi(\tau) \in \Sym \SC_{Q}$ fixes exactly 
$n/4$ elements in $\SC_{Q}$. 

If $Q=A_1^{24}$, then $G_2 (L)$ is isomorphic to 
the Mathieu group $M_{24}$ of degree $24$ 
by table~\eqref{table:G1G2}. 
We know from \cite[$M_{24}$]{ATLAS} that 
$M_{24}$ has exactly two conjugacy classes 
$3A$ and $3B$ of order $3$ elements;
when $M_{24}$ acts on a set of $24$ elements 
(such as $\SC_{Q}$) nontrivially, 
an element of $3A$ (resp., $3B$) fixes exactly $6$ elements 
(resp., $0$ element) in the set. 
If $\Phi(\tau)$ is contained in $3B$, then 
it follows from Lemma~\ref{lem:rank} that 
$\rank L^{\tau}= \rank A_{1}^{8}= 8$, 
which contradicts \eqref{eq:tau}. 
Thus we get $\Phi(\tau) \in 3A$, and hence 
$\Phi(\tau)$ fixes exactly $6$ elements in $\SC_{Q}$. 

If $Q=A_{2}^{12}$, then  $G_2 (L)$ is isomorphic to 
the Mathieu group $M_{12}$ of degree $12$ (see table~\eqref{table:G1G2}), 
which has exactly two conjugacy classes $3A$ and $3B$ 
consisting of elements of order $3$ (see \cite[$M_{12}$]{ATLAS});
when $M_{12}$ acts on a set of $12$ elements 
(such as $\SC_{Q}$) nontrivially, 
an element of $3A$ (resp., $3B$) fixes exactly $3$ elements 
(resp., $0$ element) in the set. 
The same argument as above shows that 
$\Phi(\tau) \in 3A$, and hence $\Phi(\tau)$ 
fixes exactly $3$ elements in $\SC_{Q}$. 

If $Q=A_3^8$, then $G_2(L)$ is isomorphic to 
$AGL_{3}(2)$ by table~\eqref{table:G1G2}, 
which is of type $2^{3}.PSL_{2}(7)$ (see \cite[p.\,408]{CS}), 
where $PSL_{2}(7)$ is the projective special linear group of 
degree $2$ over $\BF_{7}$. 
We see, by using the computer program ``MAGMA'' and \cite{ATLAS} for example, 
that this group has a unique conjugacy class 
of order $3$ elements, and that 
when this group acts on a set of $8$ elements 
(such as $\SC_{Q}$) nontrivially, 
an element of the conjugacy class 
fixes exactly $2$ elements in the set. 

If $Q = D_{6}^{4}$ or $E_{6}^{4}$, 
then $G_{2}(L)$ is isomorphic to $\FS_{4}$ 
by table~\eqref{table:G1G2}. 
The group $\FS_{4}$ has a unique conjugacy class 
of order $3$ elements which consists of all $3$-cycles. 
Hence, $\Phi(\tau)$ fixes exactly $1$ element in $\SC_{Q}$. 

If $Q=A_{6}^{4}$, then $G_2 (L)$ is isomorphic to $\FA_4$ 
by table~\eqref{table:G1G2}. The group $\FA_{4}$ has 
exactly two conjugacy classes of order $3$ elements,
which contain the $3$-cycles $(123) \in \FA_{4}$ and $(132) \in \FA_{4}$, 
respectively. Hence, $\Phi(\tau)$ fixes exactly $1$ element in $\SC_{Q}$. 

Next, we see from Lemma~\ref{lem:rank} that
\begin{equation} \label{eq:rank1}
6\BZ \ni 
\rank L^{\tau} 
  = \underbrace{\sum_{m=1}^{k} \rank Q_{m}}_{=6} + 
  \underbrace{\sum_{m=3k+1}^{n} \rank Q_{m}^{\tau}}_{\le 6}, 
\end{equation}
and hence $\rank L^{\tau}=6$ or $12$. 
If $\rank L^{\tau}=6$ (resp., $=12$), then 
$\rank Q_{m}^{\tau} = 0$ (resp., $= \rank Q_{m}$) 
for all $3k+1 \le m \le n$, which implies that 
$\tau$ acts on all of $Q_{m}$'s, $3k+1 \le m \le n$, 
fixed-point-freely (resp., trivially). 
Thus we have proved part (1). 

(2) Part (1) shows that $\#\CC_{0}=\#\CC_{18}=0$ 
in all the cases in this proposition, and also that 
$\#\CC_{6}=0$ if $Q$ is neither $A_{2}^{12}$ nor $E_{6}^{4}$.
Let us show that
\begin{equation} \label{eq:CC}
\begin{cases}
\# \CC_{12} \ge 1 & 
  \text{if $Q = A_1^{24}$, $A_3^8$, $D_6^4$, $A_2^{12}$, or $E_{6}^{4}$}, \\[1.5mm]
\# \CC_{12} \ge 2 & 
  \text{if $Q = A_{6}^{4}$}, \\[1.5mm]
\# \CC_{6} \ge 1 &
  \text{if $Q = A_2^{12}$ or $E_{6}^{4}$}.
\end{cases}
\end{equation}
We see from the proof of part (1) above that 
if $Q = A_1^{24}$, $A_3^8$, $D_6^4$, $A_2^{12}$, or $E_{6}^{4}$ 
(resp., $Q=A_{6}^{4}$), then $G_{2}(L) \cong \Phi(\Aut L) \ (\sbg \Sym \SC_{Q})$
has a unique conjugacy class $\ol{C}$ 
(resp., exactly two conjugacy classes $\ol{C}_{1}$ and $\ol{C}_{2}$) of 
order $3$ elements which fix exactly $k=n/4$ elements in $\SC_{Q}$. 
Here we observe that $G_{1}(L)=1$ or $\BZ_{2}$ 
in all the cases in this proposition (see table~\eqref{table:G1G2}), and 
recall that the sequence in \eqref{eq:exact} is exact. 
Let $C$ (resp., $C_{1}$ and $C_{2}$) be the conjugacy class 
of order $3$ elements in $H(L)=\Aut L \cap (G_{1}(Q):G_2(Q))$
corresponding to $\ol{C}$ (resp., $\ol{C}_{1}$ and $\ol{C}_{2}$) 
under the canonical projection 
$H(L) \twoheadrightarrow H(L)/G_{1}(L) \cong G_{2}(L)$ 
(see Lemma~\ref{lem:Z2}). 
Take an arbitrary element $\sigma \in C \subset H(L)$ 
(resp., $\sigma \in C_{p}$ with $p=1,\,2$). 
It is obvious that $\Phi(\sigma)$ fixes 
exactly $k=n/4$ elements in $\SC_{Q}$, and 
so we may assume that $\Phi(\sigma)$ acts on $\SC_{Q}$ 
as \eqref{eq:Qm}. 
Because the automorphism group of an indecomposable component of $Q$
does not have a Dynkin diagram automorphism of order $3$, we see that 
$\sigma|_{Q_{m}}=\id$ for all $3k+1 \le m \le n$. 
Thus we get $\rank L^{\sigma} = 12 \in 6\BZ$ 
(see \eqref{eq:rank1} above), which implies that 
$\#\CC_{12} \ge 1$ in these cases (resp., 
$\#\CC_{12} \ge 2$ since an element of $C_{1}$ is 
not conjugate to an element of $C_{2}$ also in $\Aut L$). 
In addition, if $Q=A_{2}^{12}$ or $E_{6}^{4}$, 
then for every $3k+1 \le m \le n$, 
there exists an element $w_{m} \in W(Q_{m})$ 
which acts on $Q_{m}$ fixed-point-freely 
(see Lemma~\ref{lem:fixedr}). 
Then, $\sigma':=\bigl(\prod_{m=3k+1}^{n}w_{m}\bigr)\sigma 
\in \Aut L$ satisfies \eqref{eq:tau} with $\rank L^{\sigma'}=6$. 
Thus, $\#\CC_{6} \ge 1$ if $Q=A_{2}^{12}$ or $E_{6}^{4}$. 

%%%%%%%%%%%%%%%%%%%%%%%%%%%%%
Next, we prove that the equalities hold in all of inequalities \eqref{eq:CC}; 
we show the equalities only for $\# \CC_{12}$ since the equality for $\# \CC_{6}$ 
can be shown similarly. 
Assume that $\tau \in \Aut L$ satisfies 
\eqref{eq:tau} with $\rank L^{\tau}=12$; 
it suffices to show that $\tau$ is conjugate to an element in $D \subset H(L)$, 
where we set
\begin{equation*}
D:=\begin{cases}
C 
 & \text{if $Q = A_1^{24}$, $A_3^8$, 
   $D_6^4$, $A_2^{12}$, or $E_{6}^{4}$}, \\[1.5mm]
C_{1} \sqcup C_{2}
 & \text{if $Q=A_{6}^{4}$}. 
\end{cases}
\end{equation*}
Write $\tau$ as: 
$\tau=\tau_{0}\tau_{H}$ with 
$\tau_{0} \in G_{0}(L) = W(Q)$ and 
$\tau_{H} \in H(L)$ (see \eqref{eq:HL}). 
We deduce from part (1) and 
the definitions of $\ol{C}$ and $\ol{C}_{p}$, $p=1,\,2$, 
that $\Phi(\tau_{H}) = \Phi(\tau) \in \ol{D} \ 
(\subset G_{2}(L) \cong \Phi(\Aut L))$, where we set
\begin{equation*}
\ol{D}:=\begin{cases}
\ol{C} 
 & \text{if $Q = A_1^{24}$, $A_3^8$, 
   $D_6^4$, $A_2^{12}$, or $E_{6}^{4}$}, \\[1.5mm]
\ol{C}_{1} \sqcup \ol{C}_{2}
 & \text{if $Q=A_{6}^{4}$}. 
\end{cases}
\end{equation*}
Because $\pi_{2}(\tau_{H})$ corresponds to $\Phi(\tau) = \Phi(\tau_{H})$ 
under the identification $G_{2}(L) \cong \Phi(\Aut L)$, 
it follows from part (1) and the argument above that 
$\pi_{2}(\tau_{H}) \in \ol{D}$, which implies that 
$\tau_{H} \in D$ by the definitions of $C$ and $C_{p}$, $p=1,\,2$. 

Now, let us write 
$\tau_{0} \in G_{0}(L) = W(Q) = \prod_{m=1}^{n} W(Q_{m})$ as:
$\tau_{0} = (x_{1},\,x_{2},\,x_{3},\,x_{4})$,
where $x_{i} \in \prod_{m=(i-1)k+1}^{ik} W(Q_{m})$ for $1 \le i \le 3$, and 
$x_{4} \in \prod_{m=3k+1}^{n} W(Q_{m})$. 
If we set 
$w:=(x_{1}^{-1},\,x_{2}^{-1},\,x_{3}^{-1},\,1)$, then 
we deduce from Lemma~\ref{lem:conj} that 
$w\tau=w\tau_{0}\tau_{H}$ is conjugate to 
$\tau=\tau_{0}\tau_{H}$. 
Thus, by replacing $\tau$ with $w\tau$, 
we may assume that $x_{i}=1$ for $1 \le i \le 3$. 
Now, because $\rank L^{\tau}= 12$, 
we see from the proof of part (1) that 
$\tau(Q_{m})=Q_{m}$ and $\rank Q_{m}^{\tau}=
 \rank Q_{m}$ for all $3k+1 \le m \le n$. 
Hence, $\tau|_{Q_{m}} = \bigl(x_{4}|_{Q_{m}}\bigr)\bigl(\tau_{H}|_{Q_{m}}\bigr) = \id$ 
for all $3k+1 \le m \le n$, which implies that $x_{4}=1$. 
Therefore we get $\tau = \tau_{H} \in D$. 
Thus we have proved part (2). 
This completes the proof of Proposition~\ref{prop:autoA1}.
\end{proof}
%
%=============================%
%     START SUBSECTION 0306   %
%=============================%
%
\subsection{Proof of Theorem \ref{thm:mainS3}\,(2): 
Case that $Q = D_{4}^{6}$.}
\label{subsec:main32-D4}

Finally, let us consider 
the case of $L=\Ni(Q)$ with $Q=D_{4}^{6}$; 
throughout this subsection, 
we use the description of the glue code $L/Q$ in \S\ref{subsec:s234}.
We should remark that $G_{1}(L) \cong \BZ_{3}$, 
and that $G_{2}(L) \cong \FS_{6} \cong G_{2}(Q)$ 
and hence $G_{2}(L) = G_{2}(Q)$. 

We divide this case into two propositions: 
in Proposition~\ref{prop:autoD4a} 
(resp., Proposition~\ref{prop:autoD4b}), 
we consider the case that $\tau \in \Aut L$ is 
contained (resp., not contained) in $G_{0}(L):G_{1}(L)$, 
or equivalently, $\Phi(\tau) = 1 \in \Sym \SC_{Q} \cong \FS_{6}$ 
(resp., $\Phi(\tau) \in \Sym \SC_{Q} \cong \FS_{6}$ is of order $3$). 
%
%%%%%%%%%%%%%%%%%%%%
%%% prop:autoD4a %%%
%%%%%%%%%%%%%%%%%%%%
%
\begin{prop}\label{prop:autoD4a} 
Let $L = \Ni(Q)$ be the Niemeier lattice with $Q=D_{4}^{6}$.
We have three automorphisms $\vp^{(6)},\,
\vp^{(3)}\omega^{(3)},\,\omega^{(6)} \in G_{0}(L):G_{1}(L)$ 
which satisfy \eqref{eq:tau}, and act on $Q=D_{4}^{6}$ as: 
\begin{align*}
(\gamma_1,\gamma_2,\gamma_3,\gamma_4,\gamma_5,\gamma_6)
& \stackrel{\vp^{(6)}}{\longmapsto}
(\vp(\gamma_1),\vp(\gamma_2),\vp(\gamma_3),
 \vp(\gamma_4),\vp(\gamma_5),\vp(\gamma_6)), \\
(\gamma_1,\gamma_2,\gamma_3,\gamma_4,\gamma_5,\gamma_6)
& \stackrel{\vp^{(3)}\omega^{(3)}}{\longmapsto}
(\vp(\gamma_1),\vp(\gamma_2),\vp(\gamma_3),
 \omega(\gamma_4),\omega(\gamma_5),\omega(\gamma_6)), \\
(\gamma_1,\gamma_2,\gamma_3,\gamma_4,\gamma_5,\gamma_6)
& \stackrel{\omega^{(6)}}{\longmapsto}
(\omega(\gamma_1),\omega(\gamma_2),\omega(\gamma_3),
 \omega(\gamma_4),\omega(\gamma_5),\omega(\gamma_6)),
\end{align*}
respectively. We have
$\rank L^{\vp^{(6)}}=0$,
$\rank L^{\vp^{(3)}\omega^{(3)}}=6$, 
$\rank L^{\omega^{(6)}}=12$, and also 
$G_{1}(L) = \Gone$.
Moreover, if $\tau \in G_{0}(L):G_{1}(L)$ satisfies \eqref{eq:tau}, 
then $\tau$ is conjugate to exactly one of the automorphisms above. 
\end{prop}

\begin{proof}
By the same argument as in \cite[\S4.2]{SS}, 
we see that $\omega^{(6)}$ preserves the glue code $L/Q$
(note that the actions of $\omega$ and $\vp$ on $D_{4}^{\ast}/D_{4}$ are same). 
Thus, $\omega^{(6)} \in \Aut L$. 
Since $\vp \in W(D_{4}) \omega$, it follows immediately that 
$\vp^{(6)}$ and $\vp^{(3)}\omega^{(3)}$ are contained in 
$G_{0}(L)\omega^{(6)}=W(Q)\omega^{(6)} \subset \Aut L$; 
indeed, $\vp^{(6)}$ and $\vp^{(3)}\omega^{(3)}$ are 
nothing but $\sigma_{2},\,\sigma_{3} \in \Aut L$ 
as in \S\ref{subsec:s234}, respectively. 
The equalities on the ranks of the fixed-point lattices
follow immediately from Lemma~\ref{lem:rank} and \eqref{eq:D4-fixed}. 
Also, $G_{1}(L)=\Gone$ is an immediate 
consequence of the fact that $G_{1}(L)=\BZ_{3}$. 

Now, let us show that 
if $\tau \in G_{0}(L):G_{1}(L)$ satisfies \eqref{eq:tau}, 
then $\tau$ is conjugate to one of $\vp^{(6)}$, 
$\vp^{(3)}\omega^{(3)}$, and $\omega^{(6)}$. 
Write $\tau$ as: $\tau=\tau_{0}\tau_{1}$ 
with $\tau_{0} \in G_{0}(L)$ and $\tau_{1} \in G_{1}(L)$.
Because $G_{1}(L) = \Gone$ as shown above, and $\tau \notin G_{0}(L)$, 
we have $\tau_{1} = \omega^{(6)}$ or $(\omega^{(6)})^{-1}$. 
%
%%%%%%%%%%%%%%
%%% c:D4-1 %%%
%%%%%%%%%%%%%%
%
\begin{claim} \label{c:D4-1}
$(\omega^{(6)})^{-1}$ is conjugate to $\omega^{(6)}$ 
in $H(L)=\Aut L \cap \bigl(G_{1}(Q):G_{2}(Q)\bigr)$. 
\end{claim}

\noindent
{\it Proof of Claim~\ref{c:D4-1}.}
We know from \cite[p.\,408 and Chapter~18, \S4, IX]{CS} that 
the glue code $L/Q$ forms the $[6,3,4]$ hexacode 
(see \cite[Chapter~3, \S2.5, (2.5.2)]{CS}). Then, 
the group $H(L)$ is isomorphic to 
the (nonsplit) group extension $3.\FS_{6}$ 
of $\FS_{6} \cong G_{2}(L)$ by $\BZ_{3} \cong G_{1}(L)$ 
as mentioned after \cite[Chapter~3, \S2.5, (66)]{CS}. 
We see from the character table of $H(L) \cong 3.\FS_{6}$ 
(which we can obtain by the computer program ``MAGMA'') 
that $H(L)$ has $3$ conjugacy classes of 
order $3$ elements, having $2$, $120$, $120$ elements, respectively. 
%
% \footnote{We can use the following command: \\[1.5mm] 
%   \texttt{CharacterTable(AutomorphismGroup(Hexacode()));}} 
%
Since $G_{1}(L)=
\bigl\{1,\,\omega^{(6)},\,(\omega^{(6)})^{-1}\bigr\}$ is 
a normal subgroup of $H(L)$, it follows immediately that 
$\bigl\{\omega^{(6)},\,(\omega^{(6)})^{-1}\bigr\}$ is 
one of the three conjugacy classes of order $3$ elements. 
Thus, $(\omega^{(6)})^{-1}$ is conjugate to $\omega^{(6)}$ 
in $H(L)$. \bqed

\vspace{3mm}

Let $h \in H(L)$ be such that $h^{-1}(\omega^{(6)})^{-1}h = 
\omega^{(6)}$. Here we should remark that $g^{-1}\tau g$ is 
contained in $G_{0}(L):G_{1}(L)$ and satisfies \eqref{eq:tau} 
for all $g \in \Aut L$, since $G_{0}(L):G_{1}(L) \lhd \Aut L$ and 
$G_{0}(L) \lhd \Aut L$. Thus, by replacing $\tau$ by 
$h^{-1}\tau h = (h^{-1}\tau_{0}h)(h^{-1}\tau_{1}h)$ 
if necessary, we may assume from the beginning that 
$\tau_{1}=\omega^{(6)}$. 

Now, let $\SC_{Q}=\bigl\{Q_{1},\,\dots,\,Q_{6}\bigr\}$. 
For each $1 \le m \le 6$, we have 
$\tau|_{Q_{m}}=(\tau_{0}|_{Q_{m}})\omega \in W(D_{4})\omega$. 
Hence, by Lemma~\ref{lem:D4}\,(1), 
$\tau|_{Q_{m}}$ is conjugate to either $\omega$ or $\vp$ 
in $P$ for each $1 \le m \le 6$.
Then we see from Lemma~\ref{lem:D4}\,(1) and \eqref{eq:D4-fixed} that 
$\rank Q_{m}^{\tau}=0$ (resp., $=2$) if and only if 
$\tau|_{Q_{m}}$ is conjugate to $\vp$ (resp., $\omega$).
Because $\rank L^{\tau} \in 6\BZ$, it can be easily checked that 
\begin{equation*}
\#\bigl\{1 \le m \le 6 \mid \rank Q_{m}^{\tau}=0\bigr\}=
0, \ 3, \ \text{or} \ 6.
\end{equation*}
If $\#\bigl\{1 \le m \le 6 \mid \rank Q_{m}^{\tau}=0\bigr\}=3$, 
then we may assume that 
\begin{equation} \label{eq:Qmtau}
\rank Q_{m}^{\tau}=
\begin{cases}
0 & \text{for $1 \le m \le 3$}, \\[1.5mm]
2 & \text{for $4 \le m \le 6$}.
\end{cases}
\end{equation}
Indeed, we first claim that

\begin{claim} \label{c:D4-2}
Let $g \in H(L)$ be such that $\pi_{2}(g) \in G_{2}(L) \cong \FS_6$
is contained in $\FA_{6} \lhd \FS_{6}$. 
Then, $g^{-1} \omega^{(6)} g = \omega^{(6)}$. 
\end{claim}

\noindent
{\it Proof of Claim~\ref{c:D4-2}.}
Remark that $H(L)$ acts on $G_{1}(L) \cong \BZ_{3}$ by conjugation 
since $G_{1}(L) \lhd H(L)$. Thus we obtain a group 
homomorphism $H(L) \rightarrow \Aut G_{1}(L) \cong \BZ_2$, 
which induces a group homomorphism 
$G_{2}(L) \cong H(L)/G_{1}(L) \rightarrow 
\Aut G_{1}(L) \cong \BZ_2$. Hence, $\FA_{6} \lhd \FS_{6} 
\cong G_{2}(L)$ is contained in the kernel of this group 
homomorphism. Thus we have proved Claim~\ref{c:D4-2}. \bqed

\vspace{3mm}

Now, let us assume that 
$\bigl\{1 \le m \le 6 \mid \rank Q_{m}^{\tau}=0\bigr\}=
 \bigl\{a,\,b,\,c\bigr\}$. 
Because the action of $\FA_{6}$ on a set of $6$ elements 
(such as $\SC_{Q}$) is $4$-transitive, 
there exists $g_{2} \in \FA_{6} \ (\lhd \FS_{6} \cong G_{2}(L))$ 
such that $g_{2}(1)=a$, $g_{2}(2)=b$, $g_{2}(3)=c$.
Let $g \in H(L)$ be such that $\pi_{2}(g)=g_{2}$; 
by Claim~\ref{c:D4-1}, we have $g^{-1}\omega^{(6)}g=\omega^{(6)}$, 
and hence 
\begin{equation*}
g^{-1} \tau g = (g^{-1}\tau_{0}g)(g^{-1}\tau_{1}g) = 
(g^{-1}\tau_{0}g)(g^{-1}\omega^{(6)}g)=
\underbrace{(g^{-1}\tau_{0}g)}_{\in G_0(L)}\omega^{(6)}.
\end{equation*}
Also, we see that 
$\rank Q_{m}^{g^{-1} \tau g}=0$ for $1 \le m \le 3$, and 
$\rank Q_{m}^{g^{-1} \tau g}=2$ for $4 \le m \le 6$. 
Thus, by replacing $\tau$ by $g^{-1}\tau g$, 
we may assume \eqref{eq:Qmtau}. 

We see from Lemma~\ref{lem:D4}\,(2) that 
for each $1 \le m \le 6$, 
if $\rank Q_{m}^{\tau}=0$ (resp., $=2$), or equivalently, 
if $\tau|_{Q_{m}}$ is conjugate to $\vp$ (resp., $\omega$), 
then there exists $y_{m} \in W(D_{4})$ such that 
$y_{m}^{-1}\bigl(\tau|_{Q_{m}}\bigr)y_{m}=\vp$ 
(resp., $\omega$); set $y := \prod_{m=1}^{6} y_{m} \in G_{1}(L)$. 
Then, for each $1 \le m \le 6$, 
\begin{equation*}
(y^{-1}\tau y)|_{Q_{m}}=
y_{m}^{-1}\bigl(\tau|_{Q_{m}}\bigr)y_{m}=
\begin{cases}
\vp & \text{if $\rank Q_{m}^{\tau}=0$}, \\[1.5mm]
\omega & \text{if $\rank Q_{m}^{\tau}=2$}.
\end{cases}
\end{equation*}
By \eqref{eq:Qmtau}, we see that 
$y^{-1}\tau y$ is equal to 
$\omega^{(6)}$, $\vp^{(3)}\omega^{(3)}$, or 
$\vp^{(6)}$. Thus we have proved the proposition. 
\end{proof}
%
%%%%%%%%%%%%%%%%%%%%
%%% prop:autoD4b %%%
%%%%%%%%%%%%%%%%%%%%
%
\begin{prop}\label{prop:autoD4b} 
Let $L = \Ni(Q)$ be the Niemeier lattice with $Q=D_{4}^{6}$.
We have two automorphisms $\sigma,\,\sigma' \in 
\Aut L \setminus \bigl(G_{0}(L):G_{1}(L)\bigr)$
which satisfy \eqref{eq:tau}, and act on $Q=D_{4}^{6}$ as: 
\begin{align*}
(\gamma_1,\gamma_2,\gamma_3,\gamma_4,\gamma_5,\gamma_6)
& \stackrel{\sigma'}{\longmapsto}
(\psi(\gamma_1),\vp(\gamma_2),\vp^{-1}(\gamma_3),
 \gamma_6,\vp^{-1}(\gamma_4),\vp(\gamma_5)), \\
(\gamma_1,\gamma_2,\gamma_3,\gamma_4,\gamma_5,\gamma_6)
& \stackrel{\sigma}{\longmapsto}
(\gamma_1,\omega(\gamma_2),\omega^{-1}(\gamma_3),
 \gamma_6,\omega^{-1}(\gamma_4),\omega(\gamma_5)), 
\end{align*}
respectively. We have
$\rank L^{\sigma'}=6$ and 
$\rank L^{\sigma}=12$.
Moreover, if $\tau \in \Aut L \setminus \bigl(G_{0}(L):G_{1}(L)\bigr)$ 
satisfies \eqref{eq:tau}, then 
$\tau$ is conjugate to $\sigma$ or $\sigma'$ as above. 
\end{prop}

%%%%%%%%%%%%%
\begin{proof}
The map $\sigma'$ is nothing but $\sigma_{4} \in \Aut L$ 
(see \S\ref{subsec:s234}). 
Because $G_{0}(L)=W(Q)=\prod_{m=1}^{6} W(Q_{m}) \sbg \Aut L$, and 
because $\vp \in W(D_{4})\omega$ and $\psi \in W(D_{4})$, 
we see that $\sigma \in W(Q) \sigma' \subset \Aut L$. 
Since $\sigma$ is a composition of Dynkin diagram automorphisms 
and a permutation of components, we see that 
$\sigma \in G_{1}(Q):G_{2}(Q)$ with the notation in \S\ref{subsec:Niemeier}. 
Hence, by \eqref{eq:HL}, we have $\sigma \in H(L)$.
The equalities on the ranks of the fixed-point lattices
follow immediately from Lemma~\ref{lem:rank}, along with 
\eqref{eq:D4-fixed}. 

Now, let us show that if 
$\tau \in \Aut L \setminus 
\bigl(G_{0}(L):G_{1}(L)\bigr)$ 
satisfies \eqref{eq:tau}, then 
$\tau$ is conjugate to either $\sigma$ or $\sigma'$. 
Since $G_{2}(L) \cong \FS_{6}$, we see that 
$\Phi(\tau)$ acts on the set $\SC_{Q}$ (of $6$ elements) 
as a $3$-cycle or a product of two mutually commutative $3$-cycles;
in the former case (resp., the later case), $\Phi(\tau)$ fixes 
$3$ elements (resp., $0$ element) in $\SC_{Q}$.
If $\Phi(\tau)$ fixes no element in $\SC_{Q}$, then 
it follows from Lemma~\ref{lem:rank} that 
$\rank L^{\tau}=\rank D_{4}^{2}=8$, 
which contradicts \eqref{eq:tau}. 
Thus we conclude that $\Phi(\tau)$ fixes $3$ elements, 
that is, $\Phi(\tau)$ acts on $\SC_{Q}$ as a $3$-cycle.

Write $\tau$ as: $\tau=\tau_{0}\tau_{H}$ with 
$\tau_{0} \in G_{0}(L)$ and $\tau_{H} \in H(L)$; 
note that $\tau_{H}$ is of order $3$, and 
$\Phi(\tau_{H})$ acts on $\SC_{Q}$ as a $3$-cycle. 
Recall that $H(L)$ has $3$ conjugacy classes of 
order $3$ elements, having $2$, $120$, $120$ elements, 
respectively (see the proof of Claim~\ref{c:D4-1} in 
the proof of Proposition~\ref{prop:autoD4a}). 
Furthermore we see from the character table of 
$H(L) \cong 3.\FS_{6}$ that one of 
these conjugacy classes (having 120 elements) 
consists of all order $3$ elements which 
act on $\SC_{Q}$ as $3$-cycles; 
notice that $\sigma \in H(L)$ is contained in this conjugacy class.
Thus, $\tau_{H}$ is conjugate to the $\sigma$ above in $H(L)$. 
Because $G_{0}(L) \lhd \Aut L$, 
we may assume from the beginning that $\tau_{H}=\sigma$. 

Now, let $\SC_{Q}:=\bigl\{Q_{1},\,\dots,\,Q_{6}\bigr\}$. 
Because $\tau=\tau_{0}\tau_{H}=\tau_{0}\sigma$, 
it follows from Lemma~\ref{lem:rank}, 
along with Lemma~\ref{lem:D4}\,(1) and \eqref{eq:D4-fixed}, that
\begin{equation*}
6\BZ \ni 
\rank L^{\tau} = \rank Q^{\tau} = 
\underbrace{\rank Q_{1}^{\tau}}_{=2\,\text{or}\,4} + 
\underbrace{\rank Q_{2}^{\tau}}_{=0\,\text{or}\,2} + 
\underbrace{\rank Q_{3}^{\tau}}_{=0\,\text{or}\,2} + 
\underbrace{\rank Q_{4}}_{=4}
\end{equation*}
Therefore, 
$\bigl(
 \rank Q_{1}^{\tau},\,
 \rank Q_{2}^{\tau},\,
 \rank Q_{3}^{\tau}
\bigr)
=(2,0,0)$ or $(4,2,2)$. 
Let us verify that $\tau$ is conjugate to $\sigma'$ 
in the former case; it can be shown similarly that 
$\tau$ is conjugate to $\sigma$ 
in the latter case. 
Observe that $\tau|_{Q_{1}} \in P$ (resp., 
$\tau|_{Q_{2}} \in P$, $\tau|_{Q_{3}} \in P$) 
is conjugate to $\psi$ (resp., $\vp$, $\vp^{-1}$) in $P$.
By Lemma~\ref{lem:D4}\,(2), there exists 
$y_{1} \in W(D_{4}) \cong G_{0}(Q_{1})$ 
(resp., $y_{2} \in W(D_{4}) \cong G_{0}(Q_{2})$, 
$y_{3} \in W(D_{4}) \cong G_{0}(Q_{3})$) such that 
$y_{1}^{-1}\bigl(\tau|_{Q_{1}}\bigr)y_{1}=\psi$ 
(resp., $y_{2}^{-1}\bigl(\tau|_{Q_{2}}\bigr)y_{2}=\vp^{-1}$, 
$y_{3}^{-1}\bigl(\tau|_{Q_{3}}\bigr)y_{3}=\vp$).
Set $y:=\prod_{m=1}^{3} y_{m} \in G_{0}(L)$. 
Then we see that $(y^{-1}\tau y)|_{Q_{m}}=\sigma'|_{Q_{m}}$
for $1 \le m \le 3$. Furthermore, we deduce by Lemma~\ref{lem:conj} 
(see also the argument at the end of 
the proof of Proposition~\ref{prop:autoA5}) that
this $y^{-1}\tau y$ is conjugate to $\sigma'$. 
Thus we have proved Proposition~\ref{prop:autoD4b}. 
\end{proof}

Combining Propositions~\ref{prop:autoD4a} and 
\ref{prop:autoD4b}, we see that $\#\CC_{0}=1+0=1$, 
$\#\CC_{6}=1+1=2$, $\#\CC_{12}=1+1=2$, 
$\#\CC_{18}=0+0=0$ in the case of $Q=D_{4}^{6}$. 
This completes the proof of 
Theorem \ref{thm:mainS3}\,(2).
%
%=========================%
%     START SECTION 04    %
%=========================%
%
\section{Review on Miyamoto's $\BZ_3$-orbifold construction.}
\label{sec:review}

In this section, we review lattice VOAs, 
twisted modules over lattice VOAs, 
and Miyamoto's $\BZ_3$-orbifold construction; 
for details, see \cite[\S6.4 and \S6.5]{LL} (and also 
\cite[\S2.1]{SS}), \cite{L, DL} (and also \cite[\S2.2]{SS}), 
and \cite{M} (and also \cite[\S2.3]{SS}), respectively.
Here we use the notation in \cite[\S2]{SS}. 
%
%==============================%
%     START SUBSECTION 0401    %
%==============================%
%
\subsection{Lattice VOAs.}
\label{subsec:lattice}

Let $L$ be a positive-definite, even lattice 
with $\BZ$-bilinear form $\pair{\cdot\,}{\cdot}$, 
and let $V_L:= M(1) \otimes _{\BC} \BC \{ L \}$ be 
the lattice VOA associated with $L$, with the vertex operator
\begin{equation*} 
Y(\cdot\,,\,z) : 
V_L \rightarrow (\End_{\BC} V_L)[\hspace{-1.5pt}[ z,z^{-1} ]\hspace{-1.5pt}],\ 
a \mapsto Y(a,z) = \sum _{n \in \BZ} a_n z^{-n-1}, 
\end{equation*}
where $M(1)$ is the free boson VOA associated to 
$\Fh := L \otimes _{\BZ} \BC$ (regarded as 
an abelian Lie algebra), 
and $\BC \{ L \}$ is the twisted group ring of $L$ 
(for details, see \cite[\S2.1]{SS}).
Recall that 
$V_L$ is spanned by the elements of the form:
$h_{k}(-n_{k}) \cdots h_{1}(-n_{1}) 1 \otimes e^{\alpha}$ 
with $h_{1},\,\dots,\,h_{k} \in \Fh$, 
$n_{1},\,\dots,\,n_{k} \in \BZ_{ > 0}$, and $\alpha \in L$; 
the weight of this element is equal to 
\begin{equation*}
n_{k}+\cdots+n_{1}+\frac{\pair{\alpha}{\alpha}}{2} \in \BZ_{\ge 0}.
\end{equation*}
In particular, the weight one subspace $(V_L)_1$ of $V_{L}$ is 
spanned by $\bigl\{h(-1) 1 \otimes e^0 \mid h \in \Fh \bigr\} \cup 
\bigl\{1 \otimes e^\alpha \mid \alpha \in \Delta \bigr\}$, 
where $\Delta := 
\bigl\{\alpha\in L \mid \pair{\alpha}{\alpha} =2 \bigr\}$. 
%
%==============================%
%     START SUBSECTION 0402    %
%==============================%
%
\subsection{Twisted modules over lattice VOAs.}

Let $L$ be a positive-definite, even lattice
with $\BZ$-bilinear form $\pair{\cdot\,}{\cdot}$. 
If $\tau \in \Aut L$ is of odd order, then 
there exists a $\tau$-invariant $2$-cocycle 
$\ve_{0}:L \times L \rightarrow \BZ/s\BZ$, where $s=2|\tau|$ 
(see, for example, the argument at the beginning of \cite[\S2.2]{SS}).
Hence each $\tau \in \Aut L$ of odd order 
induces a VOA automorphism of $V_L$, 
denoted also by $\tau$, of the same order such that 
%
%%%%%%%%%%%%%%%%%%
%%% eq:tau-VOA %%%
%%%%%%%%%%%%%%%%%%
%
\begin{equation} \label{eq:tau-VOA}
\tau \bigl( h_k (-n_k) \cdots h_1 (-n_1)1 \otimes e^{\alpha} \bigr) 
 := (\tau h_k)(-n_k) \cdots (\tau h_1)(-n_1)1 \otimes e^{\tau \alpha}
\end{equation}
for $h_{1},\,\dots,\,h_{k} \in \Fh$, 
$n_{1},\,\dots,\,n_{k} \in \BZ_{ > 0}$, and $\alpha \in L$. 

Assume that $L$ is a Niemeier lattice, and 
$\tau \in \Aut L$ is of order $3$. 
Since $V_{L}$ is holomorphic and $C_{2}$-cofinite, 
we see from \cite[Theorem~10.3]{DLM00} that 
there exists a unique irreducible $\tau$-twisted 
$V_{L}$-module, which we denote by $V_{L}(\tau)$.
We know from \cite{L, DL} (see also \cite[\S2.2]{SS}) 
the following realization of $V_{L}(\tau)$. 
Let $\zeta$ be a primitive third root of unity, and 
set $\Fh_{(m)}=\bigl\{ h \in \Fh \mid \tau (h) = \zeta^m h \bigr\}$ 
for $m \in \BZ$; note that $\Fh_{(m)}=\Fh_{(m+3)}$ 
for every $m \in \BZ$.
Define the $\tau$-twisted affinization 
$\ha{\Fh}[\tau]$ of $\Fh$ and 
its Lie subalgebra $\ha{\Fh}[\tau]_{\ge 0}$ by 
\begin{align*}
& \ha{\Fh}[\tau ]:= 
 \bigoplus _{n \in (1/3)\BZ} 
 \bigl(\Fh _{(3n)} \otimes _{\BC} \BC t^n\bigr) \oplus \BC \bk
 \quad \text{(with $\bk$ a central element)}, \\
& \ha{\Fh}[\tau] _{\ge 0} := 
 \bigoplus _{n \in (1/3)\BZ_{\ge 0}} 
 \bigl(\Fh _{(3n)} \otimes _{\BC} \BC t^n\bigr) \oplus \BC \bk,
\end{align*}
respectively, and then define 
the ``$\tau$-twisted'' free bosonic space $M(1)[\tau] := 
\Ind_{\ha{\Fh}[\tau]_{\ge 0}}^{\ha{\Fh}[\tau]}\BC$. 
Furthermore, following \cite{L, DL} (see also \cite[\S2.2]{SS}), 
we define a certain central extension $\ha{L}_{\tau}$ of $L$ 
by the cyclic group $\kc$ of order $2|\tau|=6$. 
Let $N := \bigl\{ \alpha \in L \mid 
\pair{\alpha}{\Fh_{(0)}} = \{ 0 \} \bigr\}$, and 
$\ha{N}_{\tau} \sbg \ha{L}_{\tau}$ the inverse image 
of $N \subset L$ under the canonical projection 
$\ha{L}_{\tau} \twoheadrightarrow L$.
By \cite{L, DL}, 
there exists a unique finite-dimensional, 
irreducible $\ha{N}_{\tau}$-module $T(\tau)$ such that 
$M(1)[\tau] \otimes_{\BC} U(\tau)$, 
with $U(\tau):=\Ind_{\ha{N}_{\tau}}^{\ha{L}_{\tau}} T(\tau)$, 
can be endowed with an irreducible $\tau$-twisted $V_{L}$-module 
structure; we have $V_{L}(\tau) \cong M(1)[\tau] \otimes_{\BC} U(\tau)$ 
by the uniqueness and irreducibility of $\tau$-twisted modules.

The $\tau$-twisted vertex operator for 
$V_L (\tau)$ is denoted by  
\begin{equation*}
Y_{\tau} (\cdot\,,\,z) : 
V_{L} \rightarrow 
(\End _{\BC} V_{L}(\tau))[\hspace{-1.5pt}[ z^{1/3},\,z^{-1/3} ]\hspace{-1.5pt}], \quad
a \mapsto Y_{\tau}(a,\,z) = 
\sum _{n \in (1/3)\BZ} a_n z^{-n-1}.
\end{equation*}
Notice that $V_{L}(\tau)$ is spanned by the elements of the form:
$h_k (-n_k) \cdots h_1 (-n_1)1 \otimes (g \cdot t)$ with 
$n_1,\,\ldots,\,n_k \in (1/3)\BZ_{>0}$, 
$h_1 \in \Fh_{(-3n_1)},\,\ldots,\,h_k \in \Fh_{(-3n_k)}$, 
$g \in \ha{L}_{\tau}$, and $t \in T(\tau)$; 
the weight of this element is equal to 
%
%%%%%%%%%%%%%%%%
%%% eq:wt-tw %%%
%%%%%%%%%%%%%%%%
%
\begin{equation} \label{eq:wt-tw}
n_k + \cdots + n_1 + 
 \frac{\pair{\ol{g}_{(0)}}{\ol{g}_{(0)}}}{2} + \rho, 
\end{equation} 
where
\begin{align} \label{eq:topwt}
\rho := 
 \frac{1}{18}(\dim \Fh_{(1)} + \dim \Fh_{(2)}) = 
 \frac{1}{18}(\rank L - \rank L^{\tau}),
\end{align}
the map $\ol{\,\cdot\,}:\ha{L}_{\tau} \twoheadrightarrow L$ is 
the canonical projection from $\ha{L}_{\tau}$ onto $L$, and 
for $h \in \Fh$ and $m \in \BZ$, 
$h_{(m)} \in \Fh_{(m)}$ denotes the image of $h$ 
under the orthogonal projection from $\Fh$ onto $\Fh_{(m)}$.
Remark that $\rho$ is the top weight of $V_{L}(\tau)$, that is, 
$V_L(\tau) = \bigoplus_{n \in (1/3)\BZ_{\ge 0}} 
V_L(\tau)_{n+\rho}$. 

%==============================%
%     START SUBSECTION 0403    %
%==============================%
%
\subsection{Miyamoto's $\BZ_{3}$-orbifold construction.}
\label{subsec:M-Z3}

Let $L$ be a Niemeier lattice, and let $\tau \in \Aut L$ 
be such that $|\tau|=3$ and $\rank L^{\tau} \in 6\BZ$; 
by \eqref{eq:topwt}, for each $r=1,\,2$, 
the top weight $\rho$ of the irreducible 
$\tau^{r}$-twisted $V_{L}$-module $V_{L}(\tau^{r})$ 
is equal to $1/3$ (resp., $2/3$, $1$, $4/3$) 
if $\rank L^{\tau}=18$ (resp., $12$, $6$, $0$). 
Set $V_{L}(\tau^{r})_{\BZ} := 
\bigoplus _{n \in \BZ} V_L(\tau^{r})_n$ for $r=1,\,2$, 
and then define 
%
%%%%%%%%%%%%
%%% eq:M %%%
%%%%%%%%%%%%
%
\begin{equation} \label{eq:M}
\ti{V}_{L}^{\tau}:=
 V_L^\tau\oplus V_L(\tau)_\BZ\oplus V_L(\tau^2)_\BZ,
\end{equation}
where $V_{L}^{\tau}$ is the fixed-point subVOA of $V_{L}$ 
under $\tau \in \Aut V_{L}$. 
We know the following theorem from \cite[\S3]{M}. 
%
%%%%%%%%%%%%%
%%% thm:M %%%
%%%%%%%%%%%%%

\begin{thm} \label{thm:M}
Keep the notation and setting above. 
We can give $\ti{V}_{L}^{\tau}$ a VOA structure of 
central charge $24=\rank L$. Furthermore, 
$\ti{V}_{L}^{\tau}$ is holomorphic and $C_{2}$-cofinite. 
\end{thm}
%
%%%%%%%%%%%%%%%
%%% rem:SCE %%%
%%%%%%%%%%%%%%%
%
\begin{rem} \label{rem:SCE}
\mbox{}
\begin{enu}

\item The holomorphic VOA $\ti{V}_{L}^{\tau}=
V_L^\tau \oplus V_L(\tau)_\BZ \oplus V_L(\tau^2)_\BZ$ is 
a $\BZ_{3}$-graded, simple current extension of 
the $\tau$-fixed subVOA $V_{L}^{\tau}$ of $V_{L}$; 
for the definition and properties of simple current extensions, 
see \cite[\S2]{LY} for example.
Thus the linear automorphism $\phi$ of 
$\ti{V}_{L}^{\tau}$ defined by: $\phi|_{V_L^\tau}=1$, 
$\phi|_{V_L(\tau)_\BZ}=\zeta$, and 
$\phi|_{V_L(\tau^{2})_\BZ}=\zeta^{2}$ is 
a VOA automorphism of $\ti{V}_{L}^{\tau}$. 

\item 
Let $\sigma,\,\tau \in \Aut L$ be of order $3$. 
If $\sigma$ and $\tau$ are conjugate to 
each other in $\Aut L$, then 
$\sigma,\,\tau \in \Aut V_{L}$ are also conjugate 
to each other in $\Aut V_{L}$. Indeed, 
we see from \cite[Theorem~2.1]{DN} that
$\sigma$, $\tau$ are contained in $O(\ha{L}) = 
\Hom(L,\,\BZ_2).\Aut L \sbg \Aut V_{L}$. 
So it suffices to show that every element in 
$\Hom(L,\,\BZ_2) \sigma$ of order $3$ is 
conjugate to each other. Let $x\sigma \in 
\Hom(L,\,\BZ_2) \sigma$ be of order $3$, 
with $x \in \Hom(L,\,\BZ_2)$, $x \ne 1$; 
we show that $x\sigma$ is conjugate to $\sigma$ in $\Aut V_{L}$. 
Indeed, we see that $\sigma$ is of order $3$, and 
$(x\sigma^{-1})^{3}=1$. Hence, 
$(\sigma x \sigma^{-1})^{-1} (x\sigma) (\sigma x \sigma^{-1})
 =\sigma(x\sigma^{-1})^{3} = \sigma$. Thus we have shown 
that $x\sigma$ is conjugate to $\sigma$ in $\Aut V_{L}$, as desired.
\end{enu}
\end{rem}

We denote by
\begin{equation*}
\ti{Y}(\cdot\,,\,z) : 
\ti{V}_{L}^{\tau} \rightarrow 
(\End _{\BC} \ti{V}_{L}^{\tau}) [\hspace{-1.5pt}[z,\,z^{-1}]\hspace{-1.5pt}], \quad
a \mapsto \ti{Y}(a,\,z) = \sum _{n \in \BZ} a_n z^{-n-1}
\end{equation*}
the vertex operator for the VOA $\ti{V}_{L}^{\tau}$;
remark that for $a \in V_{L}^{\tau}$, 
\begin{equation*}
\ti{Y}(a,\,z) =
\begin{cases}
Y(a,\,z) & \text{on $V_L ^{\tau}$}, \\
Y_{\tau}(a,\,z) & \text{on $V_L(\tau)_{\BZ}$}, \\
Y_{\tau^2}(a,\,z) & \text{on $V_L(\tau ^2)_{\BZ}$}.
\end{cases}
\end{equation*}
%
%%%%%%%%%%%%%%%
%%% lem:Ysh %%%
%%%%%%%%%%%%%%%
%
\begin{lem}\label{lem:Ysh}
Keep the notation and setting above. 
We set
\begin{align*}
 & \FH_{0}:=\bigl\{ h(-1)1 \otimes e^{0} \mid 
   h \in \Fh_{(0)} \bigr\} \subset (V_{L}^{\tau})_{1}; \\
 & \FH_{1}:=\bigl\{ h(-1/3)1 \otimes t \mid 
   h \in \Fh_{(-1)},\, t \in T(\tau) \bigr\}
   \subset (V_{L}(\tau))_{\rho+1/3}; \\
 & \FH_{2}:=\bigl\{ h(-1/3)1 \otimes t \mid 
   h \in \Fh_{(-2)},\, t \in T(\tau^{2}) \bigr\}
   \subset (V_{L}(\tau^{2}))_{\rho+1/3}.
\end{align*}
Let $a \in \FH_{0}$. Then the $0$-th operator 
$a_{0} \in \End_{\BC} V_{L}$ acts on $\FH_{0}$ trivially. 
Also, for each $r=1,\,2$, the $0$-th operator 
$a_{0} \in \End_{\BC} V_{L}(\tau^{r})$ acts on both of 
$\FH_{r}$ and the top weight subspace 
$V_{L}(\tau^{r})_{\rho}$ trivially. 
\end{lem}

\begin{proof}
Let $a \in \FH_{0}$ and $r=1,\,2$. 
First, it is obvious from the definition of 
the vertex operator on $V_{L}$ (or, is well-known) that 
$a_{0} \in \End_{\BC} V_{L}$ acts on $\FH_{0}$ trivially. 
Next we know from \cite[Lemma~2.2.2\,(1)]{SS} that 
$a_{0} \in \End_{\BC} V_{L}(\tau^{r})$ acts on 
$V_{L}(\tau^{r})_{\rho}$ trivially. 
Finally, it follows immediately from 
\cite[(2.2.8) and (2.2.9)]{SS} that 
$a_{0} \in \End_{\BC} V_{L}(\tau^{r})$ acts on $\FH_{r}$ 
trivially. Thus we have proved the lemma.
\end{proof}
%
%==============================%
%     START SUBSECTION 0404    %
%==============================%
%
\subsection{Lie algebra of the weight one subspace.}
\label{subsec:level}

Let $V=\bigoplus _{n \in \BZ_{\ge 0}} V_n$ be an arbitrary VOA, 
with $Y(\cdot\,,\,z):V \rightarrow 
(\End_{\BC} V)[\hspace{-1.5pt}[z,z^{-1}]\hspace{-1.5pt}]$, $a \mapsto Y(a,\,z) = 
\sum_{n \in \BZ} a_{n} z^{-n-1}$, as the vertex operator. 
If $\dim V_0 =1$, then the weight one subspace $V_1$ 
has a Lie algebra structure with the Lie bracket 
defined by $[a,b] := a_0 b$ for $a,b \in V_1$. 
When the Lie algebra $V_{1}$ is a semisimple Lie algebra, 
we define the level of a simple component of $V_{1}$ as follows. 
Assume that $\Fs \subset V_{1}$ is a simple ideal of type $X_{m}$. 
Let $\kappa_{\Fs}(\cdot\,,\,\cdot)$ be the Killing form of $\Fs$ 
normalized so that the norm of a long root of $\Fs$ is equal to $2$. 
Then there exists $\ell_{\Fs} \in \BC$ such that 
for every $x,\,y \in \Fs$ and $u,\,v \in \BZ$, 
\begin{align} \label{eq:level}
[x_u,\,y_v] = (x_0 y)_{u+v} + \ell_{\Fs} u \delta_{u+v,\,0}
\kappa_{\Fs}(x,\,y) \id_{V} \qquad \text{in $\End_{\BC} V$}. 
\end{align}
We call $\ell_{\Fs}$ the level of $\Fs$, and say that 
$\Fs$ is of type $X_{m,\,\ell_{\Fs}}$

Now, keep the notation and setting in \S\ref{subsec:M-Z3}.
Since the VOA $\ti{V}_{L}^{\tau}$ in Theorem~\ref{thm:M}
satisfies $\dim (\ti{V}_{L}^{\tau})_{0} = \dim (V_{L}^{\tau})_{0} = 1$, 
the weight one subspace $(\ti{V}_{L}^{\tau})_{1}$ has 
a Lie algebra structure. 
Because $\ti{V}_{L}^{\tau}$ is holomorphic and $C_{2}$-cofinite, 
it follows immediately from \cite[Theorem~3]{DM04} that 
the Lie algebra $(\ti{V}_{L}^{\tau})_{1}$ is either
$\{ 0 \}$, the abelian Lie algebra of dimension $24$, 
or a semisimple Lie algebra of rank less than or equal to $24$.
%
%%%%%%%%%%%%%%%%%%
%%% rem:Z3grad %%%
%%%%%%%%%%%%%%%%%%
%
\begin{rem} \label{rem:Z3grad}
For simplicity of notation, we often set
%
%%%%%%%%%%%%
%%% eq:g %%%
%%%%%%%%%%%%
%
\begin{equation} \label{eq:g}
\Fg:=(\ti{V}_{L}^{\tau})_{1} = 
 \underbrace{(V_{L}^{\tau})_{1}}_{=:\Fg_0} \oplus 
 \underbrace{V_{L}(\tau)_{1}}_{=:\Fg_1} \oplus 
 \underbrace{V_{L}(\tau^{2})_{1}}_{=:\Fg_2}.
\end{equation}
Because $\ti{V}_{L}^{\tau}=
V_{L}^{\tau} \oplus V_{L}(\tau)_{\BZ} \oplus V_{L}(\tau^{2})_{\BZ}$ 
is a $\BZ_{3}$-grading of the VOA $\ti{V}_{L}^{\tau}$ 
(see Remark~\ref{rem:SCE}\,(1)), 
we see that $\Fg=\Fg_{0} \oplus \Fg_{1} \oplus \Fg_{2}$ 
is a $\BZ_{3}$-grading of the Lie algebra $\Fg$. 
Furthermore, the restriction of 
the VOA automorphism $\phi \in \Aut \ti{V}_{L}^{\tau}$ 
(see Remark~\ref{rem:SCE}\,(1)) to the Lie algebra $\Fg$ 
is nothing but the Lie algebra automorphism corresponding 
to the $\BZ_{3}$-grading (see \cite[\S8.1]{Kac}).
\end{rem}
%
%=========================%
%     START SECTION 05    %
%=========================%
%
\section{VOA structure of $\ti{V}_L^{\tau}$.}
\label{sec:LieStructure}
%
%=============================%
%     START SUBSECTION 0501   %
%=============================%
%
\subsection{Main result in \S\ref{sec:LieStructure}.}
\label{subsec:main4}
By the following theorem, 
we conclude that the VOAs obtained in \cite{M} and 
\cite{SS} are all of non-lattice VOAs 
which we can obtain by applying Miyamoto's 
$\BZ_{3}$-orbifold construction to 
a Niemeier lattice and its automorphism. 
%
%%%%%%%%%%%%%%%%%
%%% thm:main4 %%%
%%%%%%%%%%%%%%%%%
%
\begin{thm} \label{thm:main4}
Let $L$ be a Niemeier lattice, 
and let $\tau \in \Aut L$ be such that $|\tau|=3$ 
and $\rank L^{\tau} \in 6\BZ$. 
Let $\ti{V}_{L}^{\tau}$ be the holomorphic VOA 
obtained by applying Theorem~\ref{thm:M} to 
these $L$ and $\tau$.

\begin{enu}

\item
If $\tau$ is contained in the Weyl group $G_{0}(L)$, 
then $\ti{V}_{L}^{\tau}$ is isomorphic to the lattice VOA
associated to a Niemeier lattice. 

\item
Assume that $L=\Lambda$, the Leech lattice; 
note that $\rank \Lambda^{\tau} \in \bigl\{0,\,6,\,12\bigr\}$ 
by table \eqref{table:order3}. 
If $\rank \Lambda^{\tau}=0$, then $(\ti{V}_{\Lambda}^{\tau})_{1}=
\{0\}$ (see also Remark~\ref{rem:moonshine} below). 
Otherwise, $\ti{V}_{\Lambda}^{\tau} \cong V_{\Lambda}$. 

\item
Assume that $L \ne \Lambda$ and $\tau \notin G_{0}(L)$; 
note that $\rank L^{\tau} \in \bigl\{0,\,6,\,12\bigr\}$
by table \eqref{table:order3}. 
\begin{enumerate}[\rm (3a)]

\item
If $\rank L^{\tau}=0$ or $6$, then 
$\ti{V}_{L}^{\tau}$ is isomorphic to one of 
the holomorphic non-lattice VOAs obtained in \cite{M} and \cite{SS}. 

\item
If $\rank L^{\tau}=12$, then 
$\ti{V}_{L}^{\tau} \cong V_{L}$. 

\end{enumerate}

\end{enu}
\end{thm}

\begin{rem} \label{rem:moonshine}
If $L=\Lambda$ and $\rank \Lambda^\tau=0$, then 
$\ti{V}_\Lambda^\tau$ would be isomorphic to 
the Moonshine VOA $V^{\natural}$ (see \cite[\S3.1]{M}). 
\end{rem}
%
%=============================%
%     START SUBSECTION 0502   %
%=============================%
%
\subsection{Proof of Theorem~\ref{thm:main4}\,(1) -- 
case of $\tau \in G_0 (L)$.}
\label{subsec:weyl}
We first assume that $L$ is a positive-definite, even lattice. 
Let $V_{L}$ be the lattice VOA associated to $L$.
For each $a \in (V_{L})_1$, $\exp a_{0}$ is 
a VOA automorphism of $V_{L}$ (see \cite[\S2.3]{DN}), 
where $a_0 \in \End _{\BC} V_L$ denotes the $0$-th 
operator of $a \in V_{L}$. Set
\begin{equation*}
G: = \bigl\langle \exp a_{0} \mid a \in (V_{L})_1 \bigr\rangle 
 \sbg \Aut V_{L} ;
\end{equation*}
notice that the restriction of an element in $G$ to $(V_{L})_{1}$ is 
an inner automorphism of the Lie algebra $(V_{L})_{1}$ 
in the sense of \cite[\S2.3]{H}. 

The next lemma and Lemma~\ref{L2} are well-known 
(or easy exercises for experts), 
but we give proofs for them for completion.
%
%%%%%%%%%%
%%% L1 %%%
%%%%%%%%%%
%
\begin{lem} \label{L1}
Keep the notation and setting above. Let $\tau \in G$ be of finite order. 
Then the $\tau$-fixed subVOA $V_{L}^{\tau}$ of $V_{L}$ 
is isomorphic to a lattice VOA. 
\end{lem}

\begin{proof} 
We first remark that $(V_L)_1$ is reductive.
By \cite[Proposition 8.1]{Kac}, 
there exists a Cartan subalgebra $\Fh'$ of $(V_L)_1$ 
such that $\tau=\exp h_0$ for some $h \in \Fh'$.
Since Cartan subalgebras of $(V_L)_1$ are conjugate 
under $G$, there exists $g \in G$ such that $g(\Fh')$ 
is identical to the canonical Cartan subalgebra 
$\bigl\{h(-1)1 \otimes e^{0} \mid h \in \Fh\bigr\}$ of $(V_L)_1$. 
Set $\tau'=g\tau g^{-1}=\exp g(h)_0$. 
Since $g(h)$ is contained in the canonical Cartan subalgebra above, 
we deduce that $\tau'$ acts on $M(1)\otimes e^\beta$ 
as the scalar multiple by $\exp\,\pair{g(h)}{\beta}$ 
for each $\beta \in L$.
Because $|\tau'|=|\tau| < \infty$, 
it follows immediately that 
$\bigl(\exp\,\pair{g(h)}{\beta}\bigr)^{|\tau|}=1$ 
for every $\beta \in L$.
Set $v:=|\tau|g(h)/2\pi\sqrt{-1}$; 
since $\exp\,\pair{g(h)}{\beta} = 
\exp\,\bigl(2\pi\sqrt{-1}\pair{v}{\beta}/|\tau|\bigr)$ for $\beta \in L$, 
we see that $\pair{v}{\beta} \in \BZ$ for every $\beta \in L$, 
and that $\tau'$ acts trivially on $M(1) \otimes e^\beta$ 
if and only if $\pair{v}{\beta} \in |\tau|\BZ$.
So, let us set 
$J:=\bigl\{\beta \in L \mid 
\pair{v}{\beta} \in |\tau|\BZ \bigr\} \subset L$; 
clearly it is a sublattice of $L$.
Since $\pair{v}{L} \subset \BZ$ as seen above, we have $|\tau|L \subset J$, 
and hence $J \otimes_{\BZ} \BC = L \otimes_{\BZ} \BC=\Fh$. 
Therefore we conclude that $V_L^{\tau'}=V_{J}$. 
Since $\tau'$ is conjugate to $\tau$ under $G \sbg \Aut V_{L}$ 
by the definition, it follows immediately that 
$V_L^\tau \cong V_L^{\tau'}$. Combining these, we obtain 
$V_L^\tau \cong V_{J}$, thereby completing 
the proof of the lemma. 
\end{proof}
%
%%%%%%%%%%%%%%%%%%
%%% rem:index3 %%%
%%%%%%%%%%%%%%%%%%
%
\begin{rem}\label{rem:index3} 
It is well-known that if $V_L\cong V_{L'}$ then $L\cong L'$. 
Also, if $V_L^\tau\cong V_J$, then $J$ is isomorphic 
to a sublattice of $L$ with $\#(L/J)=|\tau|$.
\end{rem}
%
%%%%%%%%%%
%%% L2 %%%
%%%%%%%%%%
%
\begin{lem} \label{L2} 
Let $J$ be a positive-definite, even lattice. 
If $U$ is a simple current extension of the lattice VOA $V_{J}$, 
then $U$ is isomorphic to the lattice VOA associated to 
a sublattice of $J^{\ast}$. 
\end{lem}

\begin{proof} 
By \cite{D}, every irreducible $V_J$-module is 
isomorphic to $V_{\lambda+J}$ for some $\lambda+J\in J^*/J$.
Hence there exists a subset $S \subset J^*/J$ 
such that $U \cong \bigoplus_{\lambda+J\in S} V_{\lambda+J}$.
By the fusion product $V_{\lambda+J} \boxtimes V_{\mu+J} \cong 
V_{\lambda+\mu+J}$ (see \cite[Corollary 12.10]{DL93}), 
we deduce that the subset $S$ is a subgroup of $J^*/J$.
Hence there exists a sublattice $M \subset J^*$ such that $S=M/J$.
By the uniqueness of simple current extensions 
(see \cite[Proposition 5.3]{DM04b}), we have $U \cong V_M$ as VOAs.
Thus we have proved the lemma.
\end{proof}
Combining these lemmas, we obtain the following proposition. 
%
%%%%%%%%%%%%%%%%%%%%%
%%% prop:orbifold %%%
%%%%%%%%%%%%%%%%%%%%%
%
\begin{prop} \label{prop:orbifold}
Let $L$ be a positive-definite, even lattice, and 
let $\tau \in G$ be of finite order. 
If $\ti{V}_L^\tau$ is a simple current 
extension of $V_{L}^{\tau}$, 
then it is isomorphic to a lattice VOA.
\end{prop}

\begin{proof}[Proof of Theorem~\ref{thm:main4}\,(1)]
We deduce from 
\cite[Lemma 3.8]{Kac} and \cite[Lemma 2.5]{DN} that 
$\tau \in G=\langle \exp a_{0} \mid a \in (V_{L})_1 \rangle$.  
It follows immediately from Remark~\ref{rem:SCE}\,(1) and 
Proposition~\ref{prop:orbifold} that 
the VOA $\ti{V}_{L}^{\tau}$ is isomorphic to a lattice VOA. 
Because the central charge of $\ti{V}_{L}^{\tau}$ 
is equal to $24$, the rank of the lattice is equal to $24$. 
Furthermore, because $\ti{V}_{L}^{\tau}$ is holomorphic, 
it follows immediately that the lattice is unimodular. 
Hence the lattice is a Niemeier lattice. 
Thus we have proved Theorem~\ref{thm:main4}\,(1).
\end{proof}
%
%=============================%
%     START SUBSECTION 0503   %
%=============================%
%
\subsection{Proof of Theorem~\ref{thm:main4}\,(2) -- 
case of the Leech lattice.}
\label{subsec:leech}

Recall that the Leech lattice $\Lambda$ is 
a unique Niemeier lattice whose root lattice $Q$ is 
identical to $\bigl\{ 0 \bigr\}$.

Now, let us start to prove Theorem~\ref{thm:main4}\,(2). 
The assertion for the case of $\rank \Lambda^{\tau}=0$ 
has been proved in \cite[\S3.1]{M}.
Assume that $\rank \Lambda^{\tau}=6$ or $12$.
For simplicity of notation, set
\begin{equation*}
\Fg := (\ti{V}_\Lambda^\tau)_1 = 
\underbrace{(V_\Lambda^\tau)_1}_{=:\Fg_0} \oplus 
\underbrace{V_{\Lambda} (\tau)_1}_{=:\Fg_1} \oplus 
\underbrace{V_{\Lambda} (\tau^{2})_1}_{=:\Fg_2}; 
\end{equation*}
note that $\Fg_0= 
\bigl\{ h(-1)1 \otimes e^{0} \mid h \in \Fh_{(0)} \bigr\} 
= \FH_{0}$ (with notation in Lemma~\ref{lem:Ysh}) 
since $Q = \{ 0 \}$. Therefore, $\Fg_0$ 
is an abelian Lie subalgebra of $\Fg$ by Lemma~\ref{lem:Ysh}, and 
$\dim \Fg_0 = \rank \Lambda^\tau \in \bigl\{ 6,\,12 \bigr\}$. 

We first assume that $\rank \Lambda^\tau=6$. 
Then we see from \eqref{eq:topwt} that the top weights of 
$V_{\Lambda} (\tau)$ and $V_{\Lambda} (\tau^{2})$ 
are both equal to $1$, and hence 
$\Fg_{1}$ and $\Fg_{2}$ are the top weight subspaces of 
$V_{\Lambda} (\tau)$ and $V_{\Lambda} (\tau^{2})$, respectively. 
Therefore it follows immediately from Lemma~\ref{lem:Ysh}
that $[\Fg_{0},\,\Fg_{1}]=[\Fg_{0},\,\Fg_{2}]=\{0\}$, 
which implies that 
$\Fg_0$ is a (nontrivial) abelian ideal of $\Fg$. 
Thus we conclude by \cite[Theorem~3]{DM04} that 
$(\ti{V}_{\Lambda}^{\tau})_1$ is 
an abelian Lie algebra of rank $24$, and 
$\ti{V}_\Lambda^{\tau} \cong V_\Lambda$, as desired. 

We next assume that $\rank \Lambda^\tau=12$.
By \cite[Theorem~3]{DM04}, $\Fg$ is abelian or semisimple.
Suppose, by contradiction, that $\Fg$ is semisimple.
We deduce from \cite[(2.2.8) and (2.2.9)]{SS} 
that $\ad a = a_{0}$ is diagonalizable on $\Fg$ 
for every element $a \in \Fg_{0}=\FH_{0}$.
Thus, by \cite[Lemma 8.1\,b)]{Kac}, 
the centralizer $\Fz$ of $\Fg_0$ in $\Fg$ is 
a Cartan subalgebra of $\Fg$. 
Define $\FH_{1}$ and $\FH_{2}$ as in Lemma~\ref{lem:Ysh}; 
note that $\FH_{r} \subset (V_{L}(\tau^{r}))_{1}$ for $r=1,\,2$ 
since $\rank L^{\tau}=12$, and hence $\rho=2/3$. 
It follows immediately from Lemma \ref{lem:Ysh} that 
$\FH _1 \oplus \FH _2 \subset \Fz$, and 
hence $\FH_1 \oplus \FH _2$ is 
an abelian subalgebra of $\Fg$. 
Let $\Fg^{\alpha} \subset \Fg$ be the root space of $\Fg$ 
corresponding to $\alpha \in \Fz^{\ast}:=\Hom_{\BC}(\Fz,\,\BC)$ 
(with respect to $\Fz$). 
Then we have $[\FH _1 \oplus \FH _2,\, \Fg ^{\alpha}] = \{0\}$ 
for all $\alpha \in \Fz^{\ast}$; 
indeed, let $x \in \Fg^{\alpha}$, and let $h \in \FH_{1}$. 
If $\alpha(h)=0$, then we have $[h,\,x]=\alpha(h)x=0$. 
Assume that $\alpha(h) \ne 0$. 
Write $x \in \Fg^{\alpha}$ as: 
$x = x_0 + x_1 + x_2 \in \Fg ^{\alpha}$, 
with $x_i \in \Fg _i$ for $i=0,\,1,\,2$. 
Since both $x_{0} \in \Fg_{0}$ and 
$h \in \FH _1$ are contained 
in the Cartan subalgebra $\Fz$, 
we see that $[h,\,x_{0}]=0$. 
Since $\Fg=\Fg_{0} \oplus \Fg_{1} \oplus \Fg_{2}$ is 
a $\BZ_{3}$-grading of $\Fg$, we get 
\begin{equation*}
\alpha(h)x_{0}+\alpha(h)x_{1}+\alpha(h)x_{2}=
\alpha(h)x = [h,\,x] = 
 \underbrace{[h,\,x_{1}]}_{\in \Fg_{2}}+
 \underbrace{[h,\,x_{2}]}_{\in \Fg_{0}}.
\end{equation*}
Since $\alpha(h)$ is assumed to be nonzero, 
we obtain $x_{1}=0$. Substituting this into the equality above, 
we get $x_{0}=0$, and then $x_{2}=0$. Thus, $x=0$, 
and in particular, $[h,\,x]=0$. Similarly, we can show that 
$[h,\,x]=0$ for all $h \in \FH_{2}$. 
Therefore, $\FH _1 \oplus \FH _2$ is a (nontrivial) abelian ideal 
of $\Fg$, which contradicts the assumption that 
$\Fg$ is semisimple. Hence we conclude by \cite[Theorem 3]{DM04} 
that $\Fg$ is abelian, and $\ti{V}_\Lambda^\tau \cong V_\Lambda$, 
as desired. Thus we have proved Theorem~\ref{thm:main4}\,(2).
%
%=============================%
%     START SUBSECTION 0504   %
%=============================%
%
\subsection{Proof of Theorem~\ref{thm:main4}\,(3) -- 
case that $L \ne \Lambda$ and $\tau \notin G_{0}(L)$.}
\label{subsec:Nie3}

Recall from \S\ref{sec:Lsigma6} the classification of 
the automorphisms $\tau \in \Aut L$ 
satisfying \eqref{eq:tau}. 
If $\rank L^{\tau} = 0$ or $6$, then 
we know from Theorem~\ref{thm:mainS3}\,(2ii) 
that $\tau$ is conjugate 
to one of $\sigma_{1},\,\dots,\,\sigma_{6}$. Thus, 
Theorem~\ref{thm:main4}\,(3a) follows immediately from 
Remark~\ref{rem:SCE}\,(2). 

In order to prove Theorem~\ref{thm:main4}\,(3b), 
we need the following lemma, which can be shown 
in exactly the same way as Lemma~\ref{lem:rank}. 
%
%%%%%%%%%%%%%%%%%%
%%% lem:cyclic %%%
%%%%%%%%%%%%%%%%%%
%
\begin{lem} \label{lem:cyclic}
Let $\CL$ be a Lie algebra, and 
let $\CI_{q}$, $1 \le q \le 4$, be ideals of $\CL$ 
such that $\CL=\bigoplus_{q=1}^{4} \CI_{q}$. 
Assume that a Lie algebra automorphism $\phi \in \Aut \CL$ 
of order $3$ acts on $\CL$ as: 
$\phi(\CI_{1})=\CI_{2}$, $\phi(\CI_{2})=\CI_{3}$, 
$\phi(\CI_{3})=\CI_{1}$, $\phi(\CI_{4})=\CI_{4}$. 
Then the $\phi$-fixed Lie subalgebra $\CL^{\phi}$ 
of $\CL$ is isomorphic to $\CI_{1} \oplus \CI_{4}^\phi$.
\end{lem}

%%%%%%%%%%%%%
\begin{proof}[Proof of Theorem~\ref{thm:main4}\,(3b)]
Let $Q$ be the root lattice of $L$.
Since $\rank L^{\tau}=12$ by assumption, 
we see from table~\eqref{table:order3} that 
the type of $Q$ is one of the following: 
$A_5 ^4 D_4$ (see Proposition~\ref{prop:autoA5}), and
$A_1 ^{24}$, $A_3 ^8$, $A_6 ^4$, $D_6^4$, 
$A_2 ^{12}$, $E_6 ^4$
(see Proposition~\ref{prop:autoA1}), and
$D_4 ^6$ (see Propositions~\ref{prop:autoD4a} and 
  \ref{prop:autoD4b}). 
By these propositions, 
along with Remark~\ref{rem:SCE}\,(2), 
we may assume that $\tau$ acts on the set 
$\SC_{Q}=\bigl\{Q_{m} \mid 1 \le m \le n\bigr\}$ of 
indecomposable components of $Q$ as follows:
there exists $0 \le k \le n/3$ such that 
for $1 \le m \le n$, 
\begin{equation} \label{eq:tauQm}
\tau(Q_{m})=
 \begin{cases}
   Q_{m+k} & \text{\rm if $1 \le m \le 2k$}, \\[1.5mm]
   Q_{m-2k} & \text{\rm if $2k+1\le m \le 3k$}, \\[1.5mm]
   Q_m & \text{\rm if $3k+1 \le m \le n$}, 
 \end{cases}
\end{equation}
and for each $3k+1 \le m \le n$, 
the restriction $\tau|_{Q_{m}} \in \Aut Q_{m}$ 
of $\tau$ to $Q_{m}$ is one of the identity map, 
a conjugation of $\omega$ in $P$, and 
a conjugation of $\omega^{-1}$ in $P$ 
(for the definitions of $\omega$ and $P$, 
see \S\ref{subsec:root} and \S\ref{subsec:autQm}, respectively). 

\begin{claim} \label{c:r12-1}
The Lie algebra $\Fg_{0}:=(V_L^\tau)_1$ is 
a semisimple Lie algebra, with 
$\FH_{0} = \bigl\{ h(-1)\otimes e^0 \mid 
h \in \Fh _{(0)} \bigr\}$ as a Cartan subalgebra.
\end{claim}

\noindent
{\it Proof of Claim~\ref{c:r12-1}.}
For each $1 \le m \le n$, let $\Fg(Q_{m})$ be 
the simple ideal of $(V_{L})_{1}$ corresponding to $Q_{m}$; 
we have $(V_{L})_{1}=\bigoplus_{m=1}^{n} \Fg(Q_{m})$.
By \eqref{eq:tau-VOA}, the isomorphism $\tau$ in \eqref{eq:tauQm} 
induces an automorphism $\tau \in \Aut V_{L}$, 
which permutes the Lie subalgebras $\Fg(Q_{m})$'s as in \eqref{eq:tauQm}. 
Therefore it follows from Lemma~\ref{lem:cyclic} that
%
%%%%%%%%%%%%%%%%
%%% eq:r12-1 %%%
%%%%%%%%%%%%%%%%
%
\begin{equation} \label{eq:r12-1}
\Fg_{0}=
(V_{L}^{\tau})_{1} \cong 
\bigoplus_{m=1}^{k} \Fg(Q_{m}) \oplus 
\bigoplus_{m=3k+1}^{n} \Fg(Q_{m})^{\tau}.
\end{equation}
Because $\tau|_{Q_{m}}=\id$, $\omega$, or $\omega^{-1}$ 
(up to conjugation) for each $3k+1 \le m \le n$, 
it follows immediately that $\Fg(Q_{m})^{\tau}$ 
is either $\Fg(Q_{m})$ or the simple Lie algebra of type $G_{2}$ 
for each $3k+1 \le m \le n$. 
Thus we conclude that $\Fg_{0}=(V_{L}^{\tau})_{1}$ is semisimple. 
Also, we can easily check that $\FH_{0}$ is a Cartan subalgebra of 
$\Fg_{0}=(V_{L}^{\tau})_{1}$ since $\tau \in \Aut V_{L}$ also 
permutes the canonical Cartan subalgebras
$\bigl\{h(-1)1 \otimes e^{0} \mid 
 h \in Q_{m} \otimes_{\BZ} \BC \bigr\}$ of $\Fg_{m}$, 
$1 \le m \le n$, as in \eqref{eq:tauQm}. \bqed

\vsp

For simplicity of notation, we set
\begin{equation*}
\Fg := (\ti{V}_L^\tau)_1 = 
\underbrace{(V_L^\tau)_1}_{=\Fg_0} \oplus 
\underbrace{V_{L} (\tau)_1}_{=:\Fg_1} \oplus 
\underbrace{V_{L} (\tau^{2})_1}_{=:\Fg_2}; 
\end{equation*}
recall that 
$\Fg=\Fg_{0} \oplus \Fg_{1} \oplus \Fg_{2}$ is 
a $\BZ_{3}$-grading of $\Fg$ (see Remark~\ref{rem:Z3grad})

\begin{claim} \label{c:r12-2}
The Lie algebra $\Fg$ is a semisimple Lie algebra of rank $24$, 
and the VOA $\ti{V}_{L}^{\tau}$ is isomorphic to 
the lattice VOA associated to a Niemeier lattice. 
\end{claim}

\noindent
{\it Proof of Claim~\ref{c:r12-2}.}
By \cite[Lemma 8.1\,b)]{Kac}, along with Claim~\ref{c:r12-1}, 
the centralizer $\Fz$ of $\FH _0 \subset \Fg_{0}$ in $\Fg$  
is a Cartan subalgebra of $\Fg$. 
We see from Lemma~\ref{lem:Ysh} that 
$\Fz$ contains $\FH_1 \oplus \FH_{2}$ with notation therein; 
notice that $\FH_1 \subset V_{L}(\tau)_{1}$ and 
$\FH_2 \subset V_{L}(\tau^{2})_{1}$ since 
$\rank L^{\tau}=12$, and hence $\rho=2/3$. 
Thus we obtain
\begin{align}
\dim \Fz & \ge \dim \FH_{0} + \dim \FH_{1} + \dim \FH_{2} 
\nonumber \\
& = \dim \Fh_{(0)} + 
 \{\dim \Fh_{(2)} \times \dim T(\tau)\} + 
 \{\dim \Fh_{(1)} \times \dim T(\tau^{2})\} \nonumber \\
& \ge \dim \Fh_{(0)}+\dim\Fh_{(1)}+\dim\Fh_{(2)}=\dim \Fh = 24,
\label{eq:cartan}
\end{align}
which implies that $\rank \Fg \ge 24$; 
however, since $\rank \Fg \le 24$ by \cite[Theorem~3]{DM04}, 
we get $\rank \Fg = 24$. 
Because $\Fg$ is not abelian 
(indeed, $\Fg_{0} \subset \Fg$ 
is semisimple, and hence not abelian), 
we conclude by \cite[Theorem~3]{DM04} that 
$\Fg=(\ti{V}_{L}^{\tau})_{1}$ is 
a semisimple Lie algebra of rank $24$, and
the VOA $\ti{V}_L^\tau$ is isomorphic to 
the lattice VOA associated to a Niemeier lattice. 
Thus we have proved Claim~\ref{c:r12-2}. \bqed

\vsp

Let $M$ be the Niemeier lattice such that 
$V_{M} \cong \ti{V}_{L}^{\tau}$, with the root lattice $R$;
note that $(V_{M})_{1} \cong \Fg$. 
To complete our proof of Theorem~\ref{thm:main4}\,(3b),
we will verify that the type of 
the semisimple Lie algebra $(V_{M})_{1} \ (\cong \Fg)$, 
or equivalently, 
the type of the root lattice $R$ of $M$ is the same as 
the type of the root lattice $Q$ of $L$.

First, let us recall from Remark~\ref{rem:SCE}\,(1) that 
$V_{M} \cong \ti{V}_{L}^{\tau} = 
V_{L}^{\tau} \oplus V_{L}(\tau)_{\BZ} \oplus V_{L}(\tau^{2})_{\BZ}$ is 
a $\BZ_{3}$-grading of the VOA $V_{M} \cong \ti{V}_{L}^{\tau}$. 
Define $\phi \in \Aut V_{M}$ by: 
$\phi|_{V_{L}^{\tau}}=1$, 
$\phi|_{V_{L}(\tau)_{\BZ}}=\zeta$, and 
$\phi|_{V_{L}(\tau^{2})_{\BZ}}=\zeta^{2}$ (see Remark~\ref{rem:SCE}\,(1)).
By Remark~\ref{rem:Z3grad}, 
the restriction of $\phi$ to the Lie algebra 
$(V_{M})_{1} \cong (\ti{V}_{L}^{\tau})_{1} = \Fg$, 
denoted also by $\phi$, is nothing but 
the Lie algebra automorphism corresponding to 
the $\BZ_{3}$-grading 
$\Fg=\Fg_{0} \oplus \Fg_{1} \oplus \Fg_{2}$; 
in particular, $(V_{M}^{\phi})_{1} \cong 
\Fg^{\phi} = \Fg_{0}$. 

Next, let $(V_{M})_{1}=\bigoplus_{q=1}^{p} \Fs_{q} \ (\cong \Fg)$ 
be the unique decomposition of $(V_{M})_{1} \cong \Fg$ 
into its simple ideals (see \cite[Theorem 5.2]{H}); 
we should remark that $\Fs_{q}$'s are all of simply-laced type. 
By the uniqueness of 
the decomposition, we see that 
the Lie algebra automorphism $\phi$ above 
naturally induces a permutation on the set 
$\bigl\{\Fs_{q} \mid 1 \le q \le p\bigr\}$ 
of simple ideals; we may assume that 
%
%%%%%%%%%%%%%%%%%%%%
%%% eq:phi-sq-D4 %%%
%%%%%%%%%%%%%%%%%%%%
%
\begin{equation} \label{eq:phi-sq-D4}
\phi(\Fs_{q})=
  \begin{cases}
    \Fs_{q+r} & \text{\rm if $1 \le q \le 2r$}, \\[1.5mm]
    \Fs_{q-2r} & \text{\rm if $2r+1 \le q \le 3r$}, \\[1.5mm]
    \Fs_q & \text{\rm if $3r+1 \le q \le p$}
  \end{cases}
\end{equation}
for $1 \le q \le p$, 
with some $0 \le r \le p/3$. 
By Lemma~\ref{lem:cyclic}, 
%
%%%%%%%%%%%%%%%%%%%%
%%% eq:fg-phi-D4 %%%
%%%%%%%%%%%%%%%%%%%%
%
\begin{equation} \label{eq:fg-phi-D4}
(V_{M}^{\phi})_{1} \cong 
\bigoplus_{q=1}^{r} \Fs_q \oplus 
\bigoplus_{q=3r+1}^{p} \Fs_{q}^{\phi} \quad 
(\cong \Fg^{\phi} = \Fg_{0}).
\end{equation}

Finally, recall the definition of $0 \le k \le n/3$ 
from \eqref{eq:tauQm}. 

%%%%%%%%%%%%%%%%%%%
\paragraph{Case 1.}
%%%%%%%%%%%%%%%%%%%
%
Assume that $k = 0$. Then we have 
$Q=D_{4}^{6}$ and $\tau=\omega^{(6)}$ 
(see Proposition~\ref{prop:autoD4a}), 
and hence by \eqref{eq:r12-1},
\begin{equation*}
\Fg_{0} \cong \bigl(\Fg(D_{4})^{\omega}\bigr)^{\oplus 6} 
\cong \Fg(G_{2})^{\oplus 6}.
\end{equation*}
Combining this and \eqref{eq:fg-phi-D4}, 
we obtain $\Fs_{q} \cong \Fg(G_{2})$ 
for all $1 \le q \le r$. 
However, all $\Fs_{q}$'s for $1 \le q \le r$ are 
of simply-laced type. Hence we get $r=0$ and $p=6$. 

By \cite[Proposition~8.1]{Kac}, for each $1 \le q \le 6$, 
there exists a Cartan subalgebra $\Ft_{q}$ of $\Fs_{q}$, 
a Dynkin diagram automorphism $\psi_q$ of 
$\Fs_{q}$ preserving $\Ft_{q}$, 
and an element $h_q$ in the $\psi_{q}$-fixed Lie subalgebra 
$\Ft_{q}^{\psi_{q}} \subset \Ft_{q}$ such that 
%
%%%%%%%%%%%%%%%%
%%% eq:restq %%%
%%%%%%%%%%%%%%%%
%
\begin{equation} \label{eq:restq}
\phi|_{\Fs_{q}} = 
 \psi_q \exp \left(\frac{2\pi\sqrt{-1}}{3}h_q \right);
\end{equation}
remark that the order of $\psi_q$ is equal to $1$ or $3$, 
since so is $\phi|_{\Fs_{q}}$. 
Suppose that $\psi_q$ is the identity map 
for some $1 \le q \le 6$. 
Then we see that $\Fs_{q}^{\phi}$ is a reductive Lie algebra 
of simply-laced type (see also \cite[Lemma~8.1\,c)]{Kac}) 
since $\Fs_{q}$ is a simple Lie algebra of simply-laced type. 
However, this contradicts the fact that $\Fs_{q}^{\phi} \cong 
\Fg(G_{2})$. Thus, $\psi_q$ is of order $3$ 
for every $1 \le q \le 6$, which implies that 
$\Fs_{q}$ is of type $D_{4}$ for every $1 \le q \le 6$. 
Thus we get $M=L$, as desired. 

%%%%%%%%%%%%%%%%%%%
\paragraph{Case 2.}
%%%%%%%%%%%%%%%%%%%
%
Assume that $k > 0$. By \eqref{eq:r12-1}, we have
\begin{equation*}
\Fg_{0}=(V_{L}^{\tau})_{1} \cong 
\bigoplus_{m=1}^{k} \Fg(Q_{m}) \oplus 
\bigoplus_{m=3k+1}^{n} \Fg(Q_{m})^{\tau}.
\end{equation*}
We deduce from the definition that 
all $\Fg(Q_{m})$'s for $1 \le m \le k$ 
(resp., all $\Fg(Q_{m})^{\tau}$'s for $3k+1 \le m \le n$) 
are simple ideals of $\Fg_{0}=(V_{L}^{\tau})_{1} \subset 
V_{L}^{\tau}$ of level $3$ (resp., of level $1$).
Similarly, in \eqref{eq:phi-sq-D4}, 
observe that all $\Fs_q$'s for $1 \le q \le r$ 
(resp., $3r+1 \le q \le p$) are 
simple ideals of $(V_{M}^{\phi})_{1} \subset V_{M}^{\phi}$
of level $3$ (resp., $1$). 
Here we should recall that 
$V_{M}^{\phi} \cong V_{L}^{\tau}$ as VOAs. 
Thus we obtain $k=r$, 
and $\Fg(Q_{m}) \cong \Fs_{m}$ for all $1 \le m \le k$, 
which implies that the root lattice $R$ contains 
$\bigoplus_{m=1}^{3k} Q_{m}$ as its component. 
It follows immediately from the list of Niemeier lattices 
(see \cite[Chapter 16, Table 16.1]{CS} for example) that 
such a Niemeier lattice is unique. 
Thus we get $M=L$, as desired. 
This completes the proof of 
Theorem~\ref{thm:main4}\,(3b).
\end{proof}

{\small
\setlength{\baselineskip}{13pt}
\renewcommand{\refname}{References}

}


\begin{thebibliography}{XXXXXX}

%\bibitem[C]{C}
%R.W. Carter, Conjugacy classes in the Weyl group, 
%{\it Compositio Math.} {\bf 25} (1972), 1--59.

\bibitem[ATLAS]{ATLAS}
J.H. Conway, R.T. Curtis, S.P. Norton, R.A. Parker, and R.A. Wilson, 
Atlas of finite groups, Oxford, Oxford University Press 1985.

\bibitem[CS]{CS}
J.H. Conway and N.J.A. Sloane, 
Sphere packing, lattices and groups, 
Third edition, Grundlehren der Mathematischen Wissenschaften, 
Vol.~290, Springer-Verlag, New York, 1999.

\bibitem[D]{D}
C. Dong, Vertex algebras associated with even lattice, 
{\it J. Algebra} {\bf 160} (1993), 245--265.

\bibitem[DL1]{DL93}
C. Dong and J. Lepowsky, 
Generalized vertex algebras and relative vertex operators, 
Progress in Mathematics Vol.~112, 
Birkh\"auser Boston, Inc., Boston, MA, 1993.

\bibitem[DL2]{DL}
C. Dong and J. Lepowsky, 
The algebraic structure of relative twisted vertex operators, 
{\it J. Pure Appl. Algebra} {\bf 110} (1996), 259--295.

\bibitem[DLM]{DLM00}
C. Dong, H. Li, and G. Mason, 
Modular-invariance of trace functions 
in orbifold theory and generalized Moonshine, 
{\it Comm. Math. Phys.} {\bf 214} (2000), 1--56.

\bibitem[DM1]{DM04}
C. Dong and G. Mason, 
Holomorphic vertex operator algebras of small central charge, 
{\it Pacific J. Math.} {\bf 213} (2004), 253--266.

\bibitem[DM2]{DM04b}
C. Dong and G. Mason, 
Rational vertex operator algebras and the effective central charge, 
{\it Int. Math. Res. Not.} (2004), 2989--3008.

\bibitem[DN]{DN}
C. Dong and K. Nagatomo, 
Automorphism groups and twisted modules for lattice vertex operator algebras, 
{\it in} ``Recent developments in quantum affine algebras and related topics'', 
Contemp. Math. Vol.~248, pp.117-133, Amer. Math. Soc., Providence, RI, 1999.

\bibitem[H]{H}
J.E. Humphreys, 
Introduction to Lie algebras and representation theory, 
Graduate Texts in Mathematics Vol.~9, 
Springer-Verlag, New York--Berlin, 1978. 

\bibitem[K]{Kac}
V.G. Kac, Infinite-dimensional Lie algebras, 
Third edition, Cambridge University Press, Cambridge, 1990

\bibitem[LY]{LY}
C.H. Lam and H. Yamauchi, 
On the structure of framed vertex operator algebras and 
their pointwise frame stabilizers, 
{\it Comm. Math. Phys.} {\bf 277} (2008), 237--285. 

\bibitem[L]{L}
J. Lepowsky, Calculus of twisted vertex operators, 
{\it Proc. Nat. Acad. Sci. U.S.A.} {\bf 82} (1985), 
8295--8299.

\bibitem[LL]{LL}
J. Lepowsky and H. Li, 
Introduction to vertex operator algebras and 
their representations, Progress in Mathematics Vol.~227, 
Birkh\"auser Boston, Inc., Boston, MA, 2004. 

\bibitem[M]{M}
M. Miyamoto, A $\BZ_3$-orbifold theory of 
lattice vertex operator algebra and 
$\BZ_3$-orbifold constructions, 
{\it in} ``Symmetries, Integrable Systems and Representations'', 
Springer Proceedings in Mathematics and Statistics Vol.~40, 
pp.319--344, Springer-Verlag, London, 2013. 

\bibitem[SS]{SS}
D. Sagaki and H. Shimakura, 
Application of a $\BZ_{3}$-orbifold construction 
to the lattice vertex operator algebras associated 
to Niemeier lattices, to appear in {\it Trans. Amer. Math. Soc.}, 
arXiv:1302.4826.

\bibitem[S]{Sch}
A.N. Schellekens, Meromorphic $c=24$ conformal field theories, 
{\it Comm. Math. Phys.} {\bf 153} (1993), 159--185.

\end{thebibliography}
\end{document}